\documentclass[a4paper,11pt]{amsart}
\usepackage{amssymb}
\usepackage[alwaysadjust]{paralist}
\usepackage[capitalize]{cleveref}
\usepackage{graphicx}
\usepackage{float}
\numberwithin{equation}{section}
\newtheorem{claim}[equation]{Claim}
\newtheorem{cor}[equation]{Corollary}
\newtheorem{lem}[equation]{Lemma}
\newtheorem{prop}[equation]{Proposition}
\newtheorem{thm}[equation]{Theorem}
\theoremstyle{definition}
\theoremstyle{plain}
\newtheorem*{thm*}{Main Theorem}
\newtheorem{defn}[equation]{Definition}

\newtheorem{rem}[equation]{Remark}
\newtheorem{rems}[equation]{Remarks}
\newtheorem{que}[equation]{Question}

\newcommand{\sign}{\operatorname{sign}}

\begin{document}

\title[Length-Angle spectra]{Rigidity of Length-angle spectrum
	for closed hyperbolic surfaces}

\author{Sugata Mondal}
\address
{Indiana University,
Rawles Hall, 831 E 3rd Street,
Bloomington, Indiana\\}
\email{sumondal\@@iu.edu}

\thanks{\emph{Acknowledgments.} I would like to thank Chris Judge for all the discussions that I had with him on this problem. I am grateful to Mengda Lei and Chris Judge for proof checking preliminary versions of the paper. I am grateful to the mathematics department of Indiana University, Bloomington for their support and hospitality.}

\date{\today}

\maketitle

\begin{abstract}
The rigidity of {\it marked length spectrum} of closed hyperbolic surfaces due to Fricke-Klein \cite{F-K} has been the motivation of many different rigidity results specially for manifolds of negative curvature. From the works of Vigneras \cite{V}, Sunada \cite{Su} and many other authors this result is far from being true for the {\it unmarked length spectrum.} The purpose of this paper is to introduce a closely a related \emph{unmarked spectrum} the \emph{length-angle spectrum} and show that it determines the surface uniquely.
\end{abstract}
\section*{Introduction}
Let $S_g$ be a closed orientable Riemannian surface of genus $g$. For $g \ge 2$ there are many hyperbolic metrics that $S_g$ can be equipped with. Due to this sharp contrast to Mostow's rigidity theorem in higher dimensions it is an important question what kind of geometric information of a hyperbolic metric on $S_g$ can uniquely determine the metric (up to isometry).
\subsection*{Length spectrum rigidity}
Let $\mathcal{C}(S_g)$ be the space of closed curves on $S_g$ up to homotopy. Let $\mathcal{M}_g$ denote the {\it moduli space} of hyperbolic metrics on $S_g$ up to isometry. For $S \in \mathcal{M}_g$ and $\gamma \in \mathcal{C}(S_g)$ let $\ell_S(\gamma)$ denote the length of the closed geodesic freely homotopic to $\gamma$ on $S$. Since every closed essential curve on $S$ is freely homotopic to a unique closed geodesic, we denote by $\mathcal{C}(S)$ the space of all closed geodesics on $S$. The marked length spectrum of $S$ is defined to be the map $$\ell: \mathcal{C}(S) \to \mathbb{R}$$ which sends a geodesic in $\mathcal{C}(S)$ to its length. A classical result dating back to Fricke-Klein \cite{F-K} says that the  marked length spectrum of a hyperbolic metric on $S_g$ determines the metric. In \cite{B-K} the possibility of extensions of this rigidity result to negatively curved manifolds was first suspected. It was subsequently confirmed by Otal \cite{O} and Croke \cite{C} for surfaces (with variable negative curvature). Since then it has been extended in many different directions.  

The sequence of lengths of all closed geodesics on $S$, counting multiplicity and without any marking, is called the {\it length spectrum} of the surface. There is a well-known connection between the length and the Laplace spectrum of $S$ due to Huber \cite{Hu}: {\it The length and the Laplace spectrum of a hyperbolic metric on $S_g$ determines each other.} Rigidity results related to the length spectrum probably appeared for the first time in the work \cite{G} of I. M. Gel'fand who showed that there is no continuous family in $\mathcal{M}_g$ for which the length spectra stays the same. He further conjectured that the length spectrum determines the hyperbolic metric. The first counter examples to this conjecture appear in \cite{V}. Later Sunada \cite{Su} gave an elegant method for constructing such counter examples.

\begin{rem}
Rigidity questions related to (unmarked) length spectrum can be traced back to the question popularized by M. Kac \cite{K}: ~~ `Can one hear the shape of a drum ?' 
\end{rem}

The best possible result for the length spectrum is due to McKean \cite{M} that says that only finitely many non-isometric closed hyperbolic surfaces of a fixed genus can have the same length spectrum i.e. the map from $L: \mathcal{M}_g \to \mathbb{R}^{\mathbb{N}}$ that sends a metric to its length spectrum is `finite to one'. 
In \cite{W}, Wolpert gave another proof of this fact. He moreover showed that there is a proper analytic sub-variety $\mathcal{V}_g$ of $\mathcal{M}_g$ that contains all genus $g$ hyperbolic surfaces that are not determined by their length spectrum. Hence Gel'fand's conjecture, although false in general, is true in the generic sense.

\subsection*{Simple length spectrum}
Let $\mathcal{G}(S_g) \subset \mathcal{C}(S_g)$ be the space of simple closed curves on $S_g$ up to homotopy. $\mathcal{G}(S_g)$ has been an important object of study in the literature (see \cite{F-M}, \cite{M-M}). For $S \in \mathcal{M}_g$ let $\mathcal{G}(S) \subset \mathcal{C}(S)$ be the space of simple closed geodesics on $S$. The sequence of lengths of these simple closed geodesics on $S$, counting multiplicity, is called the {\it simple length spectrum} of $S$. 

Wolpert's proof in \cite{W} of McKean's result \cite{M} applies to the simple length spectrum providing that for a fixed $g$ there are only finitely many surfaces in $\mathcal{M}_g$ that have the same simple length spectrum. It would be very interesting to see how far can the simple length spectrum determine the surface. 
Since the whole length spectrum is not sufficient to determine the metric it is probably true that simple length spectrum is not sufficient either. The author could not locate any literature involving this question. 
\subsection*{Angles between simple closed geodesics}
Now we consider another geometric {\it collection} related to pairs of simple closed geodesics, the angles between them. Suppose $\gamma, \delta \in \mathcal{G}(S)$ intersect each other $\iota(\gamma, \delta)$ times. We consider the set  $\Theta_S(\gamma, \delta)$ of $\iota(\gamma, \delta)$ angles of intersection between $\gamma$ and $\delta$ where each angle is measured in the counter-clockwise direction from $\gamma$ to $\delta$ and the sequence is recorded along $\gamma$ as they occur.

Defined that way $\Theta_S(\gamma, \delta)$ is a point in the product of $\iota(\gamma, \delta)$ copies of $(0, \pi)$. Since we do not record the points of intersection between $\gamma$ and $\delta$ it is clear that $\Theta_S(\gamma, \delta)$ is defined only up to the action of the cyclic permutation $(1, 2, ..., \iota(\gamma, \delta))$. The set of angles in $\Theta_S(\gamma, \delta)$ forgetting the ordering and multiplicities would be denoted by $\Phi(\gamma, \delta)$.

\begin{rem}
	For a collection of simple closed geodesics and simple geodesic arcs $\alpha, \beta_1, \cdots, \beta_k$ on $S$ the above definition equally works to define $\Theta_S(\alpha, \cup_{i=1}^k \beta_i)$. When $\alpha$ is a geodesic arc we list the angles beginning from one of the end points of $\alpha$ through the other.  
\end{rem}
\begin{defn}
	A tuple $(\theta_1, \cdots \theta_l)$ is called an ordered subset of the angle set $\Theta_S(\alpha, \cup_{i=1}^k \beta_i)$ if the ordering of $\theta_i$s respect that of $\Theta_S(\alpha, \cup_{i=1}^k \beta_i)$.
\end{defn}

Angles between more general types of geodesics has been studied in the literature. One particular case is the self-intersection angles of non-simple closed geodesics. An interesting statistical behavior of these self-intersection angles was obtained by Pollicott and Sharp in \cite{P-S}. 
\subsection*{Length-Angle Spectrum}
The main object of our study in this paper is the {\it length-angle spectrum} $\mathcal{L}\Theta(S)$ that we define as the collection 
\begin{equation*}
	\{(\ell_S(\gamma), \ell_S(\delta), \Theta_S(\gamma, \delta)): \gamma, \delta \in \mathcal{G}(S)\}.
\end{equation*}
The main result of this paper is the following.
\begin{thm*}\label{main}
	Let $S, M$ be two closed hyperbolic surfaces of genus $g$ 
	such that their length-angle spectrum coincide i.e. $\mathcal{L}\Theta(S) = \mathcal{L}\Theta(M)$. Then $S$ is isometric to $M$.
\end{thm*}
We now go over the idea of the proof. To motivate ourselves we consider two simple closed geodesics $\alpha$ and $\beta$ on $S$ with $\iota(\alpha, \beta)= 1$. It is not that difficult to see that a thickened neighborhood of $\alpha \cup \beta$ in $S$ determines a unique compact one holed torus $\mathcal{T}(\alpha, \beta) \subset S$ with geodesic boundary. A simple but important observation is that the triple $(\ell_S(\beta_1), \ell_S(\beta_2), \Theta_S(\alpha, \beta))$ determines $\mathcal{T}(\alpha, \beta)$. In particular if $\alpha'$ and $\beta'$ be two simple closed geodesics on another hyperbolic surface $S'$ with
\begin{equation*}
	(\ell_S(\alpha), \ell_S(\beta), \Theta_S(\alpha, \beta)) = (\ell_{S'}(\alpha'), \ell_{S'}(\beta'), \Theta_{S'}(\alpha', \beta'))
\end{equation*}
then the corresponding compact one holed torus with geodesic boundary $\mathcal{T}(\alpha', \beta')$ in $S'$ is an isometric copy of $\mathcal{T}(\alpha, \beta)$ (see Lemma \ref{punct-torus}). In the first step towards the proof of \textbf{Main Theorem} we formulate a rigidity criteria of similar type that works for the whole surface. We consider a simple closed non-separating geodesic $\gamma_0$ on $S$ and construct a pants decomposition of $S$ such that different geodesics in the pants decomposition are distinguishable from the angles they make with $\gamma_0$. More precisely,
\begin{thm}\label{pants}
	There is a pants decomposition $\mathcal{P}_0 = \{ \alpha_i: i= 1, 2,..., 3g-3\}$ of $S$ that satisfies the following: $(1)$ $\gamma_0$ and $\alpha_i$ intersect minimally i.e. for $\alpha_i$ non-separating $\iota(\gamma_0, \alpha_i) = 1$ and for $\alpha_i$ separating $\iota(\gamma_0, \alpha_i)=2$, and $(2)$ the sets of angles $\Phi(\gamma_0, \alpha_i)$ are mutually disjoint i.e. $\Phi(\gamma_0, \alpha_i) \cap \Phi(\gamma_0, \alpha_j) = \emptyset$ for $i \neq j$.
\end{thm}
\begin{rems}
$(i)$ We shall see in the proof of this theorem that there are infinitely many such pants decompositions. 

$(ii)$ It would be interesting to see if one can construct a pants decomposition that makes distinct angles at each intersection. In our construction, for $\alpha_i$ separating, the two angles in $\Phi(\gamma_0, \alpha_i)$ may be identical.
\end{rems}
With such a pants decomposition at hand our marked rigidity result is:
\begin{thm}\label{markedrigidity}
Let $S'$ be a closed hyperbolic surface of genus $g$. Let $\gamma'_0$ be a simple closed geodesic on $S'$ and $\mathcal{P}'_0 = \{ \alpha'_i: i =1, \cdots, 3g-3 \}$ be a pants decomposition of $S'$ such that for each $i =1, 2, ..., 3g-3:$ $(i)$ $\ell_S(\alpha_i) = \ell_{S'}(\alpha'_i)$, 
$(ii)$ $\Theta_S(\gamma_0, \alpha_i) = \Theta_{S'}(\gamma'_0, \alpha'_i)$ and 
$(iii)$ $\Theta_S(\gamma_0, \cup_{i=1}^{3g-3} \alpha_i) = \Theta_{S'}(\gamma'_0, \cup_{i=1}^{3g-3} \alpha'_i)$ 
where $\gamma_0$ and $\mathcal{P}_0 = \{ \alpha_i: i= 1, 2, ..., 3g-3\}$ are as in Theorem \ref{pants}. Then $S'$ is isometric to $S$.
\end{thm}
The idea now is to find a way of extracting information about $\gamma_0$ and $\mathcal{P}_0$ from the length-angle spectrum $\mathcal{L}\Theta(S)$. For that we consider a sequence of simple closed geodesics $T_{\bar{v}_n}(\gamma_0)$ indexed by $\bar{v}_n = (v_{1, n}, \ldots, v_{3g-3, n}) \in {\mathbb{Z}_{+}^{3g-3}}$. Here $T_{\bar{v}_n}(\gamma_0)$ is the geodesic freely homotopic to the curve
\begin{equation*}
D^{v_1}_{\alpha_1}  \circ D^{v_2}_{\alpha_2} \circ \cdots \circ D^{v_{3g-3}}_{\alpha_{3g-3}}(\gamma_0),
\end{equation*}
that is obtained from $\gamma_0$ by applying $v_{i, n}$ Dehn twists to $\gamma_0$ along $\alpha_i$ for $i =1, 2, ..., 3g-3$ (as in Theorem \ref{pants}). In \S3 we show that the corresponding angle sets $\Theta_S(\gamma_0, T_{\bar{v}_n}(\gamma_0))$ in a specific manner encodes {\it a lot of} information that we need about $(\alpha_i)$. To give an idea of the type of information these angle sets encode we begin by the following.
\begin{defn}
For any finite set $A$ we denote the cardinality of $A$ by $\#|A|$. Two diverging sequences of integers $u_n$ and $v_n$ are called similar, denoted $u_n \approx v_n$, if there is a $k \in \mathbb{N}$ such that $|u_n - v_n| \le k$.
\end{defn}
Let $\alpha, \beta$ be two simple closed geodesics on $S$ with $\iota(\alpha, \beta) =1.$ Let $\beta_n = \mathcal{D}^n_\alpha(\beta)$ and let $\mathcal{C}_\alpha  \subset S$ denote the collar neighborhood around $\alpha$. Then we have the following.
\begin{thm}
Let $\Theta_S(\alpha, \beta) = (\phi)$ and $\Theta_S(\alpha|_{\mathcal{C}_{\alpha}}, \beta_n|_{\mathcal{C}_\alpha}) = ( \theta_1^n, \theta_2^n, ..., \theta_m^n )$. Then: $(1)$ there is a partial monotonicity among $\theta_j^n$:\\*
$(a)$ if $n > 0$ then
\begin{equation*}
\theta_1^n > \theta_2^n > ...> \phi < ...< \theta_{m-1}^n < \theta_m^n
\end{equation*}
$(b)$ if $n<0$ then
\begin{equation*}
\theta_1^n < \theta_2^n < ...< \phi > ...> \theta_{m-1}^n > \theta_m^n
\end{equation*}
where the angles before and after $\phi$ correspond to the intersections between $\alpha|_{\mathcal{C}_{\alpha}}$ and $\beta_n|_{\mathcal{C}_\alpha}$ in the two different halves of $\mathcal{C}_\alpha \setminus \alpha$, \\*
$(2)$ for any $\epsilon >0$: $\#|\{ i: \theta_i^n \in (\phi-\epsilon, \phi + \epsilon) \}| \approx n.$
\end{thm}
Given the above, the complete arguments of the proof go as follows. Let $S'$ be another closed hyperbolic surface of genus $g$ such that $\mathcal{L}\Theta(S) = \mathcal{L}\Theta(S').$ In particular, for each $n$ there are simple closed geodesics $\gamma'_n, \delta'_n$ on $S'$ such that
\begin{equation*}
	(\ell_S(\gamma_0), \ell_S(T_{\bar{v}_n}(\gamma_0)), \Theta_S(\gamma_0, T_{\bar{v}_n}(\gamma_0))) = (\ell_{S'}(\gamma'_n), \ell_{S'}(\delta'_n), \Theta_{S'}(\gamma'_n, \delta'_n)).
\end{equation*}
A priori the sequence $\gamma'_n$ depends on $n$ but the number of simple closed geodesics on $S'$ of length $\ell_S(\gamma_0)$ being finite, up to extracting a subsequence, we have a fixed closed geodesic $\gamma'_0 \in \mathcal{G}(S')$ such that
\begin{equation*}
(\ell_S(\gamma_0), \ell_S(T_{\bar{v}_n}(\gamma_0)), \Theta_S(\gamma_0, T_{\bar{v}_n}(\gamma_0))) = (\ell_{S'}(\gamma'_0), \ell_{S'}(\delta'_n), \Theta_{S'}(\gamma'_0, \delta'_n)).
\end{equation*}

The last part of the proof studies the sequence $\delta'_n$. Observe that, up to extracting a subsequence, these geodesics converge to a geodesic lamination $L$. With the assumption
\begin{equation*}
	\lim_{n \to \infty} \frac{v_{i+1, n}}{v_{i, n}} = 0, ~~ \text{for each} ~~ i = 1, 2, ..., 3g-2
\end{equation*}
we show, using results from \S3, that each leaf of $L$ spirals around a simple closed geodesic $\alpha'_i$ and the collection $(\alpha'_i)$ contains at least $3g-3$ simple closed geodesics. Since $\alpha'_i$s are mutually non-intersecting it follows that they form a pants decomposition $\mathcal{P}'_0$ of $S'$. Further analysis of the sequence $\delta'_n$ via the convergence $\delta'_n \to L$ provides us the following theorem that concludes the proof of the \textbf{Main Theorem} using Theorem \ref{markedrigidity}.
\begin{thm}\label{geodesic-pants}
Let $\gamma'_0$ and $\mathcal{P}'_0 = \{ \alpha'_i: i=1, \cdots, 3g-3 \}$ respectively be the simple closed geodesic on $S'$ and the pants decomposition of $S'$ as above. Then for each $i = 1, 2, ..., 3g-3$: $(i)$ $\ell_{S'}(\alpha'_i) = \ell_S(\alpha_i)$, $(ii)$ $\Theta_{S'}(\gamma'_0, \alpha'_i) = \Theta_S(\gamma_0, \alpha_i)$ and $(iii)$ $\Theta_{S'}(\gamma'_0, \cup_{i=1}^{3g-3} \alpha'_i) = \Theta_S(\gamma_0, \cup_{i=1}^{3g-3} \alpha_i)$, where $\gamma_0$ and $\mathcal{P}_0 = \{ \alpha_i: i= 1, 2, ..., 3g-3\}$ are as in Theorem \ref{pants}.
\end{thm}
\subsection*{Structure of the article}
In \S1 we recall some basic concepts and tools that we are going to use in the later sections. There we recall $(i)$ formal definition of Dehn twist and $(ii)$ the structure theorem for geodesic laminations. The next section is devoted to two rigidity results Lemma \ref{punct-torus} and Theorem \ref{pants}. We give proofs of these two results there. The next section \S3 is the most important section of this article from the technical point of view. We begin this section by recalling asymptotic growth of the intersection numbers $\iota(\gamma, \mathcal{D}_\alpha^n(\beta))$ for $\alpha, \beta, \gamma \in \mathcal{G}(S)$. We then use this asymptotic to study asymptotic growth of lengths of geodesics of the form  $\ell(\mathcal{D}_\alpha^n(\beta))$. Later we develop qualitative and asymptotic properties of angle sets $\Theta_S(\gamma, \mathcal{D}_\alpha^n(\beta))$. The next section, \S4, is devoted to the construction of the pants decomposition in Theorem \ref{pants}. We prove our \textbf{Main Theorem} in \S5. In the end we have a small appendix where we explain two small results.
\section{Preliminaries}
In this section we review some standard facts from hyperbolic geometry that will be used in our study. The area formula of a hyperbolic geodesic polygon is the simplest among these. Let $G$ be a hyperbolic geodesic $n$-gon with interior angles $\phi_1, ..., \phi_n$. Then the area $|G|$ of $G$ is given by
\begin{equation}\label{areapolygon}
  |G| = (n-2)\pi - \sum_{i=1}^{n} \phi_i.
\end{equation}
\subsection{Collars}
Let $\alpha$ be a simple closed geodesic on $S$. The Collar Theorem says that there is a {\it collar} neighborhood $\mathcal{C}_\alpha \subset S$ of $\alpha$ which is isometric to the cylinder $[-w(\alpha), w(\alpha)] \times \mathbb{S}^1$ with the metric $d{r^2} + \ell_\alpha^2 \cosh^2{r}d\theta^2$ and for any two non-intersecting simple closed geodesics $\alpha, \beta$ the collars $\mathcal{C}_\alpha, \mathcal{C}_\beta$ are mutually disjoint. The coordinates $(r, \theta)$ on $\mathcal{C}_\alpha$ via this isometry is called the {\it Fermi coordinates}. For an $x \in \mathcal{C}_\alpha$ let $(r(x), \theta(x))$ be its Fermi coordinates. Then $r(x)$ denotes the signed distance of $x$ from $\alpha$ and $\theta(x)$ denotes the projection of $x$ on $\alpha$ when $\alpha$ is identified with $\mathbb{S}^1$ \cite[p-94]{Bu2}. 
\subsection{Dehn twist}
Dehn twist homeomorphisms are the most important tools used in this paper. We use them, for example, to construct our sequence of geodesics $T_{\bar{v}_n}(\gamma_0) \in \mathcal{G}(S)$. For a more complete and detailed discussion of these we refer the readers to \cite[Chapter 3]{F-M}.

Let $\alpha$ be a simple closed geodesic on $S$ and $\mathcal{C}_\alpha \subseteq S$ be the collar neighborhood of $\alpha$. Identify $\mathcal{C}_\alpha$ with $[-w(\alpha), w(\alpha)] \times \mathbb{S}^1$ via the isometry explained above. Now consider the homeomorphism $D_\alpha$ of $[-T, T] \times \mathbb{S}^1 \subseteq \mathcal{C}_\alpha$ given by
 \begin{equation*}
   (r, \theta) \to (r, \theta + \pi - \frac{\pi}{T}\cdot r).
 \end{equation*}
Since $D_\alpha$ fixes the two boundary circles $\{-T\} \times \mathbb{S}^1$ and $\{T\} \times \mathbb{S}^1$ pointwise it can be extended to the rest of the surface as identity. This homeomorphism (up to isotopy) is called the Dehn twist around $\alpha$. For a simple closed geodesic $\beta$ by $\mathcal{D}_\alpha(\beta)$, called the Dehn twist of $\beta$ along $\alpha$, we mean the simple closed geodesic freely homotopic to $D_\alpha(\beta)$. It is a standard fact that $\mathcal{D}_\alpha (\beta) \neq \beta$ iff $\iota(\alpha, \beta) \neq 0$ \cite[Proposition 3.2]{F-M}.
\subsection{End-to-end geodesic arcs}
Consider the Fermi coordinates $(r, \theta)$ on $\mathcal{C}_\alpha$ that identifies $\mathcal{C}_\alpha$ with $[-w(\alpha), w(\alpha)] \times \mathbb{S}^1$. Now consider the curves in $\mathcal{C}_\alpha$ that are graphs of smooth maps $[-w(\alpha), w(\alpha)] \to \mathbb{S}^1$. We call them {\it end-to-end arcs}. When such a curve is a geodesic we call it an {\it end-to-end geodesic arc}. The end-to-end geodesic arc with constant $\theta$ coordinate equal to $\phi$ is called the {\it $\phi$-radial arc} and is denoted by $\eta_\phi$. An end-to-end geodesic arc that does not intersect at least one radial arc will be called an {\it almost radial arc}. For an end-to-end geodesic arc $\xi$ by $\mathcal{D}_\alpha(\xi)$ we denote the end-to-end geodesic arc that is homotopic to $D_\alpha(\xi)$ under the end point fixing homotopy.	
\begin{rem}\label{almstgeod1}
Let $\partial_1^\alpha = \{-w(\alpha)\} \times \mathbb{S}^1$ and $\partial_2^\alpha = \{w(\alpha)\} \times \mathbb{S}^1$ be the two components of $\partial{\mathcal{C}_\alpha}$. Let $s_1 =(-w(\alpha), \theta_1) \in \partial_1^\alpha$ and $s_2 =(w(\alpha), \theta_2) \in \partial_2^\alpha$. It is not difficult to observe that: (i) if the $\theta_1 = \theta_2 = \phi$ then there is exactly one almost radial arc,  $\eta_\phi$, with end points $s_i \in \partial_i^\alpha$ and (ii) if the $\theta_1 \neq \theta_2$ then there are exactly two simple paths in $\mathbb{S}^1$ that join $\theta_1$ and $\theta_2$ each of which produce exactly one almost radial arc with end points $s_i \in \partial_i^\alpha$. One of these two arcs is a Dehn twist of the other along $\alpha$.
\end{rem}
\subsubsection{Orientation of an end-to-end arc}
Observe that the end-to-end arc $D_\alpha(\eta_\phi)$ is the graph of the map $\Psi: [-w(\alpha), w(\alpha)] \to \mathbb{S}^1$ that sends $t$ to $(\phi + \pi - \frac{\pi}{T}\cdot t)$. We consider the orientation of  $\mathbb{S}^1$ such that this map is orientation preserving. Observe that this orientation does not depend on $\phi$ but depends on the Fermi coordinates. For an end-to-end arc $\xi$ in $\mathcal{C}_\alpha$ we consider the smooth function $\Psi_\xi: [-w(\alpha), w(\alpha)] \to \mathbb{S}^1$ whose graph is $\xi$. We say that the orientation of $\xi$ is positive (or negative) if $\Psi_\xi$ is orientation preserving (or orientation reversing).  
\subsection{Pants decomposition}
For us a pants decomposition of a hyperbolic surface $S$ is a collection of mutually disjoint simple closed geodesics that divide the surface into three holed spheres. In \S2 and thereafter we shall consider pants decompositions of closed hyperbolic surfaces that intersect a given simple closed non-separating geodesic minimally. Figure \ref{fig:Pants} is an example of such a hyperbolic surface $M$ of genus $3$. The simple closed geodesic $\gamma$ is given by the yellow curve and the pants decomposition $P$ by the red curves.
\begin{figure}[H]
	\includegraphics[scale=.25]{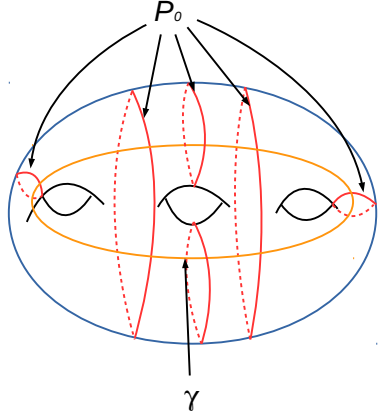}
	\caption{Pants decomposition}
	\label{fig:Pants}
\end{figure}
\subsection{Geodesic laminations}\label{gl}
A geodesic lamination on a hyperbolic surface $S$ is a closed set $\mathcal{L}$ which is a disjoint union of complete geodesics, called {\it leaves} of the lamination. The simplest examples of such things are simple closed geodesics. A little more complicated and most used one in this article are the {\it limits} of simple closed geodesics under repeated Dehn twists.

A geodesic lamination can be much more complicated than these. For a more complete and detailed discussion of this topic we refer the readers to \cite{C-M-E}. One of the most important use of these laminations come with an associated measure, and the pair is called a {\it measured geodesic lamination}. We shall not use the later in this article.
\subsubsection{A topology on $\mathcal{GL}(S)$}
Let $\mathcal{GL}(S)$ denote the space of all geodesic laminations on $S$. The Chabauty topology on $\mathcal{GL}(S)$ is the restriction of the Chabauty topology on the space $\mathcal{C}(S)$ of all closed subsets of $S$. For detailed discussion of this topology we refer to \cite[Chapter I.3.1, Chapter I.4]{C-M-E}.

\begin{rem}\label{compact}
  It is well-known that $\mathcal{GL}(S)$ is compact, separable and metrizable with respect to the Chabauty topology.
\end{rem}
\begin{defn}
 1. A subset of a geodesic lamination which itself is a geodesic lamination is called a sub-lamination. A sub-lamination is called proper if it is not the whole lamination.\\*
 2. A lamination is called minimal if it does not have any proper sub-lamination. A minimal sub-lamination of a lamination is a sub-lamination which is minimal as a lamination. \\*
 3. A leaf of a lamination is called isolated if each point on it has a neighborhood (in $S$) that do not intersect any other leaf of the lamination.\\*
 4. We say that a leaf $l$ spirals along one of its ends around a lamination $K$ if every lift of $l$ to $\mathbb{H}^2$ shares an end point (at $\partial{\mathbb{H}^2}$) with an end point of one of the leaves of a lift of $K$ to $\mathbb{H}^2$.
\end{defn}
In \S5 we would have to deal with geodesic laminations without having prior knowledge of their structure. In that situation we shall need the following structure theorem of geodesic laminations \cite[I.4.2.8. Theorem: Structure of lamination]{C-M-E} for further understanding our geodesic laminations.
\begin{thm}\label{structure}
  Let $\mathcal{L}$ be a geodesic lamination on a hyperbolic surface $S$ of finite area. Then $\mathcal{L}$ consists of disjoint union of finitely many minimal sub-laminations $K_1, ..., K_m$ and finitely many isolated leaves $l_1, ..., l_p$ such that each $l_i$ along one of its ends spirals around one of the $K_j$.
\end{thm}
For a geodesic $\gamma \in \mathcal{G}(S)$ and a geodesic lamination $L$ we can consider the angle set $\Theta_S(\gamma, L)$ and the set of angles $\Phi(\gamma, L)$ in $\Theta_S(\gamma, L)$ in exactly the same way as before. In this case however $\Theta_S(\gamma, L) \in [0, \pi]^\omega$ where $\omega$ is the cardinality of $\iota(\gamma, L)$. When the last cardinality is infinite we have accumulation points in $\Phi(\gamma, L)$.
\begin{defn}
We call a point $\phi \in \Phi(\gamma, L)$ a accumulation point of $\Theta_S(\gamma, L)$ if there is an ordered sequence $(\theta_n)$ in $\Theta_S(\gamma, L)$ that converges to $\phi$. An ordered set $(\phi_1, ..., \phi_k)$ of $\Theta_S(\gamma, L)$ is called the set of accumulation points of $\Theta_S(\gamma, L)$ if $\phi_1, ... \phi_k$ represent all accumulation points of $\Theta_S(\gamma, L)$ counting multiplicity.
\end{defn}
\subsubsection{Spiraling around a collection of geodesics}
Let $\alpha_1, \cdots, \alpha_k$ be a collection of mutually disjoint simple closed geodesics. Let $\gamma_0$ be a simple closed geodesic such that $\iota(\alpha_i, \gamma_0) \neq 0$ for each $i =1, \cdots, k$. For $(n_1, \cdots, n_ i) \in \mathbb{Z}^k$ consider the geodesic $T_{n_{1, i}, \cdots, n_{k, i}}(\gamma_0)$ that is freely homotopic to
\begin{equation*}
D^{n_1}_{\alpha_1}  \circ D^{n_2}_{\alpha_2} \circ \cdots \circ D^{n_k}_{\alpha_k}(\gamma_0).
\end{equation*}
Now let $(n_{1, i}, \cdots, n_{k, i}) \in \mathbb{Z}^k$ be a sequence such that the sign of $n_{j, i}$ be independent of $i$. Denote the sign of $n_{j, i}$ by $\sign(j)$. As $n_{j, i}$ tend to (positive or negative) infinity $T_{n_{1, i}, \cdots, n_{k, i}}(\gamma_0)$ converges to a geodesic lamination $$L_{\sign(1)\alpha_1, \cdots, \sign(k)\alpha_k}(\gamma_0)$$ that has exactly $k$ closed leaves $\alpha_1, \cdots, \alpha_k$. Rest of the leaves of this lamination are isolated and so along each of their ends they spiral around one of the $\alpha_i$. Let $\ell$ be one such half-leaf that spirals around $\alpha_i$. In \S 1.3.1 we defined the orientation of an end-to-end geodesic arc given fixed Fermi coordinates on $\mathcal{C}_{\alpha_i}$. In a very similar way we can define the orientation for $\ell$. Observe that as $\ell$ spirals around $\alpha_i$ the orientation of $\ell$ is positive if $\sign(i)$ is positive and negative if $\sign(i)$ is negative.
\section{Marked Rigidity}
In this section we prove the two rigidity results Lemma \ref{punct-torus} and Theorem \ref{markedrigidity}. These results are motivated by the marked length rigidity of closed hyperbolic surfaces due to Fricke-Klein \cite{F-K}. As warm up exercise we treat compact hyperbolic surfaces that are once holed tori with geodesic boundary. It is long know \cite{BuS} (see also \cite{H}) that for these surfaces length spectrum determines the metric. 
\begin{lem}\label{punct-torus}
	Let $\mathcal{T}$ and $\mathcal{T}'$ be two compact one holed torus with geodesic boundary. Let $(\alpha, \beta)$ and $(\alpha', \beta')$ be two pairs of simple closed geodesics on $\mathcal{T}$ and $\mathcal{T}'$ respectively with $\iota(\alpha, \beta) = 1 = \iota(\alpha', \beta')$. If
	\begin{equation*}
		(\ell_{\mathcal{T}}(\alpha), \ell_{\mathcal{T}}(\beta), \Theta_{\mathcal{T}}(\alpha, \beta)) = (\ell_{\mathcal{T}'}(\alpha'), \ell_{\mathcal{T}'}(\beta'), \Theta_{\mathcal{T}'}(\alpha', \beta'))
	\end{equation*}
	then $\mathcal{T}$ is isometric to $\mathcal{T}'$.
\end{lem}
\begin{figure}[H]
	\includegraphics[scale=.25]{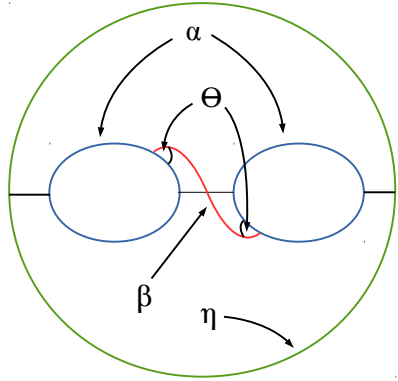}
	\caption{One holed torus}
	\label{fig:puncture}
\end{figure}
\begin{proof}
  We use cut and paste method to prove the lemma. We begin by cutting $\mathcal{T}$ along $\alpha$. This would result in a $Y$-piece $\mathcal{Y}(\ell_{\mathcal{T}}(\alpha), \ell_{\mathcal{T}}(\alpha), \ell_0)$ where the triple marks the lengths of the three boundary geodesics of the $Y$-piece. We shall first show that the third length $\ell_0$ is uniquely determined by our triple
  \begin{equation}\label{holedtorus}
    (\ell_{\mathcal{T}}(\alpha), \ell_{\mathcal{T}}(\beta), \Theta_{\mathcal{T}}(\alpha, \beta)).
  \end{equation}

  We begin by understanding what $\beta$ looks like in $\mathcal{Y}(\ell_{\mathcal{T}}(\alpha), \ell_{\mathcal{T}}(\alpha), \ell_0)$. In the above picture (Picture \ref{fig:puncture}) we consider our situation. The geodesic marked $\eta$ has length $\ell_0$. The three black arcs joining pairs of boundary geodesics of $\mathcal{Y}(\ell_{\mathcal{T}}(\alpha), \ell_{\mathcal{T}}(\alpha), \ell_0)$ are the mutual perpendiculars. The red arc that joins the two copies of $\alpha$ represents $\beta$.

  By a symmetry argument we obtain that the point of intersection of $\beta$ and the mutual perpendicular between the two copies of $\alpha$ is the mid-point of both of these geodesic arcs. Now consider one of the triangles formed by two of these arcs and one of the subarcs on (one of the copies of) $\alpha$. By the angle ratio formula from basic hyperbolic trigonometry:
  \begin{equation}\label{ypiece}
   \frac{\sinh{\ell_p}}{\sin{\theta}} = \frac{\sinh{\frac{\ell_{\mathcal{T}}(\beta)}{2}}}{\sin{\frac{\pi}{2}}} \Rightarrow \sinh{\ell_p} = \sin{\theta}.\sinh{\frac{\ell_{\mathcal{T}}(\beta)}{2}}
  \end{equation}
  where $2 \ell_p$ is the length of the mutual perpendicular between the two copies of $\alpha$ and $(\theta) = \Theta_{\mathcal{T}}(\alpha, \beta)$. Hence $\ell_p$ is determined by the data \eqref{holedtorus}.

 Now recall that any $Y$-piece is determined (up to isometry) by the lengths of its three boundary geodesics. In particular $2 \ell_p$ is a function of $\ell_{\mathcal{T}}(\alpha)$ and $\ell_0$. Fixing the value of $\ell_{\mathcal{T}}(\alpha)$ consider the map $\ell_0 \to 2 \ell_p$. A simple trigonometric computation implies that this map is injective. Hence by above $\ell_0$ is determined by \eqref{holedtorus}.
 
 Finally we get that $\mathcal{T} \setminus \alpha$ is isometric to $\mathcal{T}' \setminus \alpha'$. To conclude the theorem it suffices to observe that there is a unique way of gluing the two copies of $\alpha$ (or $\alpha'$) in the boundary of $\mathcal{Y}(\ell_{\mathcal{T}}(\alpha), \ell_{\mathcal{T}}(\alpha), \ell_0)$ (or $\mathcal{Y}(\ell_{\mathcal{T}'}(\alpha'), \ell_{\mathcal{T}'}(\alpha'), \ell_0)$) such that the red arc becomes $\beta$ (or $\beta'$) after the gluing. This follows from Lemma \ref{Ypiece}. 
\end{proof}
Now we present a proof of Theorem \ref{markedrigidity}. Let us begin by recalling it. 

\textbf{Theorem \ref{markedrigidity}}
{\it Let $S'$ be a closed hyperbolic surface of genus $g$. If there is a simple closed geodesics $\gamma'_0$ on $S'$ and a pants decomposition $\mathcal{P}'_0 = \{ \alpha'_1, ..., \alpha'_{3g-3} \}$ of $S'$ such that for each $i =1, 2, ..., 3g-3$ \\*
$(i)$ $\ell_S(\alpha_i) = \ell_{S'}(\alpha'_i)$,\\*
$(ii)$ $\Theta_S(\gamma_0, \alpha_i) = \Theta_{S'}(\gamma'_0, \alpha'_i)$ and \\*
$(iii)$ $\Theta_S(\gamma_0, \cup_{i=1}^{3g-3} \alpha_i) = \Theta_{S'}(\gamma'_0, \cup_{i=1}^{3g-3} \alpha'_i)$ \\*
where $\gamma_0$ and $\mathcal{P}_0 = \{ \alpha_i: i= 1, 2, ..., 3g-3\}$ are as in Theorem \ref{pants}	then $S'$ is isometric to $S$.}

\begin{proof}
As one moves along $\gamma_0$ the points of intersection between $\gamma_0$ and $\alpha_i$s appear one after another. The ordering of these intersections is well defined only up to a cyclic permutation. Let us assume that they appear along $\gamma_0$ in the cyclic order
\begin{equation*}
 \alpha_1 \to \alpha_2 \to ...\to \alpha_{2g-1} \to \alpha_{2g-2} \to \alpha_{2g} \to ...\to \alpha_{3g-3} \to \alpha_2 \to \alpha_1.
\end{equation*}
 In the first step we show that up to a cyclic permutation the corresponding appearance of the $\alpha'_i$s along $\gamma'_0$ is identical i.e.
 \begin{equation*}
 \alpha'_1 \to \alpha'_2 \to ...\to \alpha'_{2g-1} \to \alpha'_{2g-2} \to \alpha'_{2g} \to ...\to \alpha'_{3g-3} \to \alpha'_2 \to \alpha'_1.
\end{equation*}
To see this let $\alpha_k$ appears just after $\alpha_l$ along $\gamma_0$ and let $(\phi_1, \phi_2)$ be the ordered subset of $\Theta_S(\gamma_0, \cup_{i=1}^{3g-3} \alpha_i)$ corresponding to these two intersections. Observe that $\phi_1 \in \Theta_S(\gamma_0, \alpha_l) = \Theta_{S'}(\gamma'_0, \alpha'_l)$ and $\phi_2 \in \Theta_S(\gamma_0, \alpha_k) = \Theta_{S'}(\gamma'_0, \alpha'_k)$. Since $\Phi(\gamma'_0, \alpha'_i) \cap \Phi(\gamma'_0, \alpha'_j) = \emptyset$ for $i \neq j$ that is the only possibility as well! Hence the ordering follows.

Now let $\alpha_i, \alpha_j$ and $\alpha_k$ bound a $Y$-piece in $S$. Then there is an ordering among $\alpha_i, \alpha_j$ and $\alpha_k$ according to their appearance along $\gamma_0$. In particular, up to a change of indices, we may assume that $\alpha_i$ appears {\it exactly} before $\alpha_j$ and $\alpha_k$ appear {\it exactly} after $\alpha_j$ (observe that $\alpha_i, \alpha_j$ and $\alpha_k$ may not appear consecutively!). By the above ordering equality $\alpha'_i, \alpha'_j$ and $\alpha'_k$ appear identically along $\gamma'_0$. Since geodesics from a pants decomposition must intersect $\gamma'_0$ before any other geodesics do, we obtain that $\alpha'_i, \alpha'_j$ and $\alpha'_k$ determines a $Y$-piece. 

Recall that a hyperbolic surface can be described as a collection of marked $Y$-pieces and a set of relations for gluing pairs of identically marked boundary geodesics of these $Y$-pieces. In that setting the above basically say that $S$ and $S'$ can be constructed from identical collection of $Y$-pieces (obtained from $S \setminus \mathcal{P}_0$ or $S' \setminus \mathcal{P}'_0$) and the gluing relations possibly differ only by twists around different geodesics (in $\mathcal{P}_0$).

Next we consider a $Y$-piece $\mathcal{Y}(\alpha_i, \alpha_j, \alpha_k) \subset S$ with boundary geodesics $\alpha_i, \alpha_j$ and $\alpha_k$. Let $\mathcal{Y}'(\alpha'_i, \alpha'_j, \alpha'_k) \subset S'$ be the corresponding $Y$-piece with boundary geodesics $\alpha'_i, \alpha'_j$ and $\alpha'_k$. Since $\ell(\alpha_i) = \ell(\alpha'_i)$, $\mathcal{Y}(\alpha_i, \alpha_j, \alpha_k)$ is isometric to $\mathcal{Y}'(\alpha'_i, \alpha'_j, \alpha'_k)$ via an isometry that sends $\alpha_i \to \alpha'_i$. 
\begin{lem}\label{Ypiece}
  Let $Y$ be a pair of pants. Consider two boundary geodesics $\alpha_1, \alpha_2$ of $Y$. Let $A$ be the collection of simple geodesic arcs in $Y$ that joins them. Then every geodesic arc $\beta$ in $A$ is determined by $\Theta_Y(\beta, \alpha_1 \cup \alpha_2)$.
\end{lem}
\begin{proof}
 We argue by contradiction and assume that there are two such arcs $\beta_1, \beta_2$ with
 \begin{equation}\label{parallelangle}
  \Theta_Y(\beta_1, \alpha_1 \cup \alpha_2) = \Theta_Y(\beta_2, \alpha_1 \cup \alpha_2).
  \end{equation}
  There are two cases that need separate consideration: $(i) ~ \beta_1 \cap \beta_2 = \emptyset$ and $(ii) ~ \beta_1 \cap \beta_2 \neq \emptyset$. In the first case it is easy to observe that $\beta_1$ and $\beta_2$ along with subarcs of $\alpha_1$ and $\alpha_2$ forms a (contractible) geodesic rectangle. We reach the desired contradiction while calculating the area of this rectangle using \eqref{parallelangle}. In the second case we have a contractible triangle bounded by subarcs of $\beta_1$ and $\beta_2$ and a subarc of either $\alpha_1$ or $\alpha_2$. The area calculation of this triangle using \eqref{parallelangle} again provides the desired contradiction.
\end{proof}
\begin{cor}
Let $T$ be a hyperbolic one-holed torus with geodesic boundary. Let $\alpha$ be a simple closed geodesic on $T$. Assume that we have a simple geodesic arc $\gamma$ that joins two points on $\partial Y$ and intersects $\alpha$ exactly once. If the two angles in $\Theta_T(\gamma, \partial Y)$ are identical then $T$ has an isometry that interchanges the two points of intersection between $\partial Y$ and $\gamma$.	
\end{cor}
\begin{proof}
Cut $T$ along $\alpha$ to get the pair of pants $Y$ and denote the two copies of $\alpha$ on $\partial Y$ by $\alpha_1$ and $\alpha_2$. Consider the two components of $\gamma \setminus \alpha$. Denote the component that joins $\partial T$ and $\alpha_i$ by $\gamma_i$.

Observe that $\Theta_T(\gamma_1, \partial T \cup \alpha_1) = \Theta_T(\gamma_2, \partial T \cup \alpha_2)$. Now $Y$ has a rotational isometry around the midpoint of the mutual perpendicular between $\alpha_1$ and $\alpha_2$. By the last lemma we conclude that $\gamma_2$ is the image of $\gamma_1$ under the rotational isometry of $Y$.   
\end{proof}
Now we are ready to finish the proof of Theorem \ref{markedrigidity}. By the lemma above we observe that the isometry between $\mathcal{Y}(\alpha_i, \alpha_j, \alpha_k)$ and $\mathcal{Y}(\alpha'_i, \alpha'_j, \alpha'_k)$ actually sends
\begin{equation}
\gamma_0|_{\mathcal{Y}(\alpha_i, \alpha_j, \alpha_k)} \to \gamma'_0|_{\mathcal{Y}'(\alpha'_i, \alpha'_j, \alpha'_k)}.	
\end{equation}
Hence it suffices to prove that there is essentially a unique way of gluing the pairs of pants obtained from $\mathcal{P}_0$ or equivalently from $\mathcal{P}'_0$. The above lemma (and the assumption that $\Phi(\gamma_0, \alpha_i)$ are pairwise disjoint) imply that there essentially is no choice along non-separating $\alpha_i$s. Along a separating $\alpha_i$ there is a possibility of a twist that would interchange the two points of intersections between $\gamma_0$ and $\alpha_i$. Since the ordering of appearance of $\alpha_i$ along $\gamma_0$ is fixed this can not happen at any  but those $\alpha_i$s that bound one holed torus. By the last corollary it is clear that even if there are such choices the resulting surfaces obtained from different choices are isometric. 
\end{proof}
\section{Dehn twist: length and angles}
For the definition of Dehn twist homeomorphisms we refer the reader to \S1. This section is devoted to the understanding of the following two questions. Let $\alpha, \beta, \gamma$ be three simple closed geodesics on $S$. 
\begin{que}\label{q1}
How does the length of $\mathcal{D}^n_\alpha(\beta)$ grow with respect to the quantities $n, \iota(\alpha, \beta), \ell(\alpha)$ etc ?
\end{que}
\begin{que}\label{q2}
Is there any structural property inside an angle set and in particular in $\Theta_S(\gamma, \mathcal{D}^n_\alpha(\beta))$ ? If so how do they change with respect to $n$ ?
\end{que}
Of course we need to be more precise about the last question. We refer the reader to \S3.3 for this.
\subsection{Intersection}
To answer these two questions we shall need to know how the intersection number $\iota(\gamma, \mathcal{D}^n_\alpha(\beta))$ grow with respect to the numbers $n, \iota(\alpha, \beta),$ $\iota(\alpha, \gamma)$ and $\iota(\beta, \gamma)$. The following estimates, in this direction, are from \cite[Proposition 3.2]{F-M} and \cite[Lemma 4.2]{I}. Let $\alpha_1, ..., \alpha_k$ be $k$ mutually disjoint simple closed geodesics on $S$. For a simple closed geodesic $\gamma$ and a tuple $(n_1, ..., n_k) \in \mathbb{Z}^k$ let $T_{n_1, ..., n_k}(\gamma)$ denote the closed geodesic freely homotopic to $D^{n_1}_{\alpha_1}\circ D^{n_2}_{\alpha_2}\cdots D^{n_k}_{\alpha_k}(\gamma)$. 
\begin{prop}\label{intersection} $(1) ~~ \iota(\mathcal{D}_{\alpha_i}^n(\gamma), \alpha_i) = |n| {\iota(\gamma, \alpha_i)^2}$.
 \begin{equation*}
(2) ~~ \sum_{i=1}^k (|n_i|-2)\iota(\alpha_i, \beta)\iota(\alpha_i, \gamma) - \iota(\gamma, \beta) \le \iota(T_{n_1, ..., n_k}(\beta), \gamma) $$$$
 \end{equation*}
 \begin{equation*}
(3) ~~ \iota(T_{n_1, ..., n_k}(\beta), \gamma) \le \sum_{i=1}^k |n_i|\iota(\alpha_i, \beta)\iota(\alpha_i, \gamma) + \iota(\gamma, \beta).
 \end{equation*}
\end{prop}
One more topological result on the intersection number would be important for later development. Let $\mathcal{A}$ be either $\mathcal{C}_\alpha$ or one of the two components of $\overline{\mathcal{C}_\alpha \setminus \alpha}$. Recall that simple geodesic arcs in $\mathcal{A}$ that joins two points one in each component of $\partial{\mathcal{A}}$ are called end-to-end geodesic arcs.
\begin{lem}\label{parallel}
	Let $\beta_1, \beta_2$ be two non-intersecting end-to-end geodesic arcs in $\mathcal{A}$. Then for any end-to-end geodesic arc $\eta$ in $\mathcal{A}$ one has:
	\begin{equation*}
		|\iota(\eta, \beta_1) - \iota(\eta, \beta_2)| \le 1.
	\end{equation*}
\end{lem}
\begin{proof}
	Without loss of generality we may assume that
	\begin{equation*}
		\iota(\eta, \beta_2) \ge \iota(\eta, \beta_1) ~~ \textrm{and} ~~ \iota(\eta, \beta_2) > 1.
	\end{equation*}
	Fix two points of intersection $x, y$ between $\eta$ and $\beta_2$ that occur consecutively along $\eta$ and let $\eta'$ be the subarc of $\eta$ lying between $x$ and $y$. It suffices to prove that $\eta'$ intersects $\beta_1$ at least once. We argue by contradiction and assume that $\eta'$ and $\beta_1$ are disjoint. Cutting $\mathcal{C}_\alpha$ along $\beta_1$ we obtain a rectangle $\mathcal{R}_\alpha(\beta_1)$. Since $\beta_2$ and $\eta'$ do not intersect $\beta_1$ both of them are contained in $\mathcal{R}_\alpha(\beta_1)$. In particular we have a loop formed by $\eta'$ and the subarc of $\beta_2$ between $x$ and $y$ which is contained in $\mathcal{R}_\alpha(\beta_1)$, a topological disc. This is an impossibility.
\end{proof}
\subsection{Lengths}
 Now we consider the length of $\mathcal{D}^n_\alpha(\beta)$. The next estimate is probably well-known to experts but the author was unable to locate a reference.
\begin{prop}\label{length}
 Let $\alpha_1, \alpha_2, ..., \alpha_k$ be $k$ mutually disjoint simple closed geodesics on $S$. Then for any $\beta \in \mathcal{G}(S)$ there exist non-negative integers $k_i = {k_i}(\alpha_i, \beta)$ such that for any $(n_1, n_2, ..., n_k) \in \mathbb{Z}^k$ with $|n_i|$ sufficiently large one has:
  \begin{equation*}
  (1) \ell(T_{n_1, ..., n_k}(\beta)) \le \sum_{i=1}^{k} \iota(\alpha_i, \beta) |n_i| \ell(\alpha_i) + \ell(\beta)
  $$$$
  (2) \sum_{i=1}^{k} \iota(\alpha_i, \beta) (|n_i| - k_i) \ell(\alpha_i) \le \ell(T_{n_1, ..., n_k}(\beta)).
  \end{equation*}
\end{prop}
\begin{proof}
 For the upper bound observe that $D^{n_1}_{\alpha_1}\circ D^{n_2}_{\alpha_2}\cdots D^{n_k}_{\alpha_k}(\beta)$ is freely homotopic to the union of $\beta$ and $\iota(\alpha_i, \beta)|n_i|$ copies of $\alpha_i$ for $i=1, ..., k$ after removing the points intersection properly \cite{F-M}. Since $T_{n_1, ..., n_k}(\beta)$ is the geodesic in this free homotopy class, the upper bound follows.

 For the lower bound we consider the collar $\mathcal{C}_{\alpha_i} \subset S$ around $\alpha_i$. Since $\mathcal{C}_{\alpha_i}$ and $\mathcal{C}_{\alpha_j}$ are mutually disjoint for $i \neq j$ it suffices to consider the length of $T_{n_1, n_2, ..., n_k}(\beta)$ restricted to each $\mathcal{C}_{\alpha_i}$. By Lemma \ref{length1} in the Appendix for any simple closed geodesic $\delta$ we have:
 \begin{equation*}
 \ell(\delta|_{\mathcal{C}_{\alpha_i}}) \ge (\iota(\eta_i, \delta) - 2\iota(\alpha_i, \delta))\ell(\alpha_i),
 \end{equation*}
 where $\eta_i$ is an almost radial (see \S1.1) arc in ${\mathcal{C}_{\alpha_i}}$. So it suffices to find a simple closed geodesic whose restrictions to $\mathcal{C}_{\alpha_i}$ has at least one almost radial arc $\eta_i$ such that
 \begin{equation}\label{lower}
 \iota(T_{n_1, ..., n_k}(\beta), \eta_i) \ge \iota(\beta, \alpha_i)(|n_i| - k_i),
 \end{equation}
 for some $k_i= k(\alpha_i, \beta)$. Observe that by a similar argument as in the first paragraph (of this proof) we can easily see that for any such geodesic arc $\eta_i$ we have the upper bound
 \begin{equation}\label{upper}
 \iota(T_{n_1, ..., n_k}(\beta), \eta_i) \le \iota(\beta, \alpha_i)|n_i| + \iota(\beta, \eta_i).
 \end{equation}
 
 Let $\gamma$ be a geodesic on $S$ that intersects all the $\alpha_i$s for $i=1, ..., k$. Replacing $\gamma$ by certain combination of Dehn twists of $\gamma$ along $\alpha_i$s, if necessary, we can assume that each sub-arc of $\gamma$ in each $\mathcal{C}_{\alpha_i}$ is an almost radial arc. Applying Proposition \ref{intersection}(2) to $\gamma$ along with Lemma \ref{uniformintersection} from the Appendix we have $n = n(\gamma, \alpha_1, ..., \alpha_k)$ such that
 \begin{equation}\label{totalint}
  \sum_{i=1}^k\iota(T_{n_1, ..., n_k}(\beta)|_{\mathcal{C}_{\alpha_i}}, \gamma|_{\mathcal{C}_{\alpha_i}}) \ge \sum_{i=1}^k (|n_i| - 2)\iota(\alpha_i, \beta)\iota(\alpha_i, \gamma)$$$$ - \iota(\gamma, \beta) - n.
 \end{equation}
 To complete the proof we argue by contradiction and assume that there is a sequence $(n_{1, j}, ..., n_{k, j})_j$ such that for any geodesic arc $\eta$ appearing as subarcs of $\gamma|_{\mathcal{C}_{\alpha_1}}$ we have
 \begin{equation}\label{radial}
  \iota(T_{{n_{1, j}}, ..., {n_{k, j}}}(\beta), \eta) < (|n_{1, j}| - k(n_{1, j})) \iota(\beta, \alpha_i)
 \end{equation}
 where $k(n_{1, j}) \to \infty$ as $|n_{1, j}| \to \infty$. By \eqref{totalint} we have
 \begin{equation*}
  \sum_{i=1}^k \iota(T_{n_1, ..., n_k}(\beta)|_{\mathcal{C}_{\alpha_i}}, \gamma|_{\mathcal{C}_{\alpha_i}}) \ge \sum_{i=1}^k (|n_i| - 2) \iota(\alpha_i, \beta)\iota(\alpha_i, \gamma)$$$$ - \iota(\gamma, \beta) - n.
 \end{equation*}
 Using \eqref{radial} we get
 \begin{equation*}
  \sum_{i=2}^k \iota(T_{n_1, ..., n_k}(\beta)|_{\mathcal{C}_{\alpha_i}}, \gamma|_{\mathcal{C}_{\alpha_i}}) - \sum_{i=2}^k (|n_i| - 2)\iota(\alpha_i, \beta)\iota(\alpha_i, \gamma)$$$$  \ge (k(n_{1, i})-2)\iota(\alpha_i, \beta)\iota(\alpha_i, \gamma) - \iota(\gamma, \beta) - n \to \infty
 \end{equation*}
 as $n_{1, i} \to \infty$. This is contradictory to
 by \eqref{upper}.
 \end{proof}
 \begin{rem}\label{eachintersect}
 By the last part of the proof it follows that for any $\gamma, \beta$ and $\alpha_1, \cdots, \alpha_k$ as above we have
 \begin{equation}
 	\iota(T_{n_1, ..., n_k}(\beta)|_{\mathcal{C}_{\alpha_i}}, \gamma|_{\mathcal{C}_{\alpha_i}}) \approx |n_i|\iota(\beta, \alpha_i)\iota(\gamma, \alpha_i)
 \end{equation}
 where the implied constant may depend on the involved geodesics.
 \end{rem}
 \subsection{Angles}
 Here we study some {\it structural properties} of angle sets. In the simplest case we take two simple closed geodesics $\alpha$ and $\beta$ with $\iota(\alpha, \beta) = 1$ and consider the sequence $\beta_n= \mathcal{D}^n_{\alpha}(\beta)$. For our understanding of $\Theta_S(\alpha, \beta_n)$ it would be sufficient to understand $\Theta_S(\beta|_{\mathcal{C}_{\alpha}}, {\beta_n}|_{\mathcal{C}_\alpha})$. We begin our study by counting the number of intersections between $\beta$ and $\beta_n$ lying in the two components of $\mathcal{C}_{\alpha} \setminus \alpha$. It is reasonable to believe that these two numbers are approximately the same. Since this fact is important for us we start by giving a proof of this fact.
 \subsubsection{End-to-end geodesic arcs}
 Observe that for any simple closed geodesic $\beta$, any of its subarcs in ${\mathcal{C}_\alpha}$ is, what we called, an end-to-end geodesic arc. Recall that an end-to-end arc is a smooth simple curve on ${\mathcal{C}_\alpha}$ that are graphs of functions $[-w(\alpha), w(\alpha)] \to \mathbb{S}^1$, where we use the Fermi coordinates to identify ${\mathcal{C}_\alpha}$ with $[-w(\alpha), w(\alpha)] \times \mathbb{S}^1$.  
 
 We shall need two important but simple facts about end-toend arcs. First, for any smooth simple arc $\gamma$ in $\mathcal{C}_\alpha$ there is a unique end-to-end geodesic arc $\chi$ in $\mathcal{C}_\alpha$ that is homotopic to $\gamma$ under the end point fixing homotopy. Second, the number of intersection $\iota(\gamma_1, \gamma_2)$ between any two smooth simple arcs $\gamma_1, \gamma_2$ in $\mathcal{C}_\alpha$ is at least the number of intersection $\iota(\chi_1, \chi_2)$ between their respective geodesic representatives $\chi_1, \chi_2$ under the end point fixing homotopy. These two facts can be seen easily by taking lifts of involved curves to $\mathbb{H}^2$.
 
 Recall that for an end-to-end geodesic arc $\beta$ in $\mathcal{C}_\alpha$ by $\mathcal{D}^m_\alpha(\beta)$ we denote the geodesic freely homotopic to $D^m_\alpha(\beta)$ under the end point fixing homotopy.   
 \begin{lem}\label{endtoend}
 Let $\xi, \eta $ be two end-to-end geodesic arcs such that they have the same end points on $\partial{\mathcal{C}_\alpha}$. Then $\mathcal{D}^m_\alpha(\xi) = \eta$ for some $m \in \mathbb{Z}$. 
 \end{lem}
\begin{proof}
 It is enough to show that $D^m_\alpha(\xi)$ is homotopic to $\eta$, for some $m \in \mathbb{Z}$, under the end point fixing homotopy. To show this consider any two points of intersection between $\xi$ and $\eta$ that occur consecutively along $\eta$. Consider the piecewise geodesic loop formed by the subarc of $\xi$ and $\eta$ between these two points. This loop is freely homotopic to $\alpha$. Using the definition of the Dehn twists homeomorphism $D_\alpha$ it is not difficult to see that one of $D_\alpha(\xi)$ and $D^{-1}_\alpha(\xi)$ does not have the latter loop, up to isotopy. As a result one of them intersects $\eta$ one less number of times than $\xi$. Repeating this procedure we get $m \in \mathbb{Z}$ such that $D^m_\alpha(\xi)$ is homotopic to $\eta$.  
\end{proof}

\begin{lem}\label{intersectionequal}
  Let $\xi$ and $\eta$ be two end-to-end geodesic arcs in $\mathcal{C}_\alpha$. Then the numbers of intersection between $\xi$ and $\eta$ in the two components of $\mathcal{C}_\alpha \setminus \alpha$ differ by at most two.
\end{lem}
\begin{proof}
 We first prove the Lemma with an extra assumption that all four end points of $\xi$ and $\eta$ have the same $\theta$ coordinate equal to $\psi$. By Lemma \ref{endtoend} we know that $\xi$ and $\eta$ are Dehn twists of the $\psi$-radial arc $\eta_\psi$ (see \S1.1) of certain order i.e.
 \begin{equation*}
  \mathcal{D}^m_\alpha(\eta_\psi) = \xi ~~ \textrm{and} ~~ \mathcal{D}^n_\alpha(\eta_\psi) = \eta
 \end{equation*}
 for some $m, n \in \mathbb{Z}$.

 Now recall the Fermi coordinates $(r, \theta)$ on ${\mathcal{C}_\alpha}$. Using these coordinates consider the embedding of ${\mathcal{C}_\alpha}$ in $\mathbb{R}^3$ via the map $(r, \theta) \to (\cos{\theta}, \sin{\theta}, r)$. Recall that in the Fermi coordinates $\alpha$ is identified with $\mathbb{S}^1$. Let $a: \mathbb{S}^1 \to \mathbb{S}^1$ be the antipodal map. Now consider the line $L_\psi$ in $\mathbb{R}^3$ that intersects ${\mathcal{C}_\alpha}$ orthogonally at the points $(\cos{\psi}, \sin{\psi}, 0)$ and $(\cos{a(\psi)}, \sin{a(\psi)}, 0))$. The rotation of $\mathbb{R}^3$ by an angle of $\pi$ with axis $L_\psi$ when restricted to ${\mathcal{C}_\alpha}$ provides an isometry $R_\psi$ of ${\mathcal{C}_\alpha}$. It is easy to check that $R_\psi$ interchanges the two components of ${\mathcal{C}_\alpha} \setminus \alpha$. Moreover, using explicit expressions one can observe that $D^k_\alpha(\eta_\psi)$ are invariant under $R_\psi$. By uniqueness $\mathcal{D}^k_\alpha(\eta_\psi)$ are also invariant under $R_\psi$ for each $k$. It follows that the numbers of intersections between $\mathcal{D}^m_\alpha(\eta_\psi) = \xi$ and $\mathcal{D}^n_\alpha(\eta_\psi) = \eta$ in the two components of $\mathcal{C}_\alpha \setminus \alpha$ are identical.

 To prove the general case we first observe that for any end-to-end geodesic arc $\chi$ whose two end points have different $\theta$ coordinates one can always find another end-to-end geodesic arc $\zeta$ disjoint from $\chi$ whose both end points have the same $\theta$ coordinates. To see this we use Lemma \ref{endtoend} to express $\chi$ as the Dehn twist of an almost radial geodesic arc $\chi_0$ i.e. $\chi = \mathcal{D}^m_\alpha(\chi_0)$ for some integer $m$. Now by definition $\chi_0$ does not intersect at least one radial arc say $\eta_\phi$. Then our $\zeta$ is simply $\mathcal{D}^m_\alpha(\eta_\phi)$. The general case now follows from Lemma \ref{parallel}.
\end{proof}
\subsubsection{Angle set and intersections}
The fact that the two halves of $\mathcal{C}_\alpha \setminus \alpha$ contains approximately the same number of intersections between any two end-to-end geodesic arcs is not yet {\it visible} in their angle set. A part of our next result would make it so. Fix one set of Fermi coordinates on $\mathcal{C}_\alpha \setminus \alpha$. With respect to these coordinates consider the orientation of $\alpha$ from \S 1.3.1.

Let $\gamma$ be an end-to-end geodesic arc in $\mathcal{C}_\alpha$ with $\iota(\alpha, \gamma) = 1$ and $x$ be the point of intersection. Let $\pi: \mathbb{H}^2 \to S$ be a fixed universal covering such that $\pi(0) = x$ and denote the corresponding lifts of $\alpha$ and $\gamma$ by the same alphabet. Observe that the above orientation of $\alpha$ corresponds to the orientation of the lift of $\alpha$ that decreases the height. Consider Picture \ref{fig:twist1} where the left column corresponds to situations in $\mathcal{C}_\alpha \subset S$ and the right column corresponds to one set of lifts of the involved geodesics to $\mathbb{H}^2$. 

\begin{thm}\label{asympt}
   Let $\Theta_S(\gamma, \alpha) = (\phi)$. Let $\xi$ be another end-to-end geodesic arc with $\Theta_S(\gamma, \xi) = (\theta_1, \theta_2, ..., \theta_m)$. Then: \\*
  $(1)$ either
  \begin{equation*}
  \theta_1 > \theta_2 > ...> \phi < ...< \theta_{m-1} < \theta_m
  \end{equation*}
  or
  \begin{equation*}
  \theta_1 < \theta_2 < ...< \phi > ...> \theta_{m-1} > \theta_m
  \end{equation*}
  where the angles before and after $\phi$ correspond to the intersections between $\gamma$ and $\xi$ that occur in the two different halves of $\mathcal{C}_\alpha \setminus \alpha$, \\*
  $(2)$ for any $\epsilon >0$ the cardinality $\#|\{ j: \theta_j \in (\phi - \epsilon, \phi + \epsilon) \}| \approx m$.
  \end{thm}
\begin{figure}[H]
	\includegraphics[scale = .35]{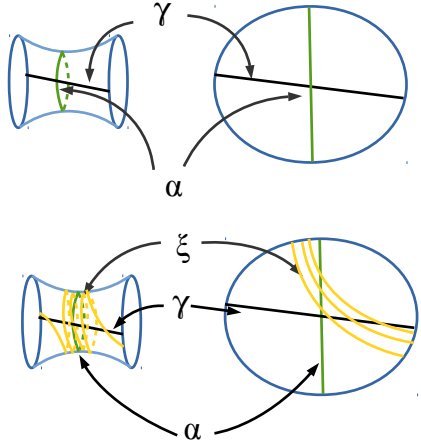}
	\caption{Twists and their lifts}
	\label{fig:twist1}
\end{figure}
 
  \begin{proof}
  We begin by making the observation that $\theta_i$ does not depend on the orientation of $\gamma$ or $\xi$. Since the result we want to prove is qualitative it is enough to consider the angles corresponding to the points of intersection between $\gamma$ and $\xi$ lying in one of the halves of $\mathcal{C}_\alpha \setminus \alpha$. Without loss of generality let us consider the right half $\{(r, \theta): 0 \le r \le w(\alpha)\}$ of $\mathcal{C}_\alpha \setminus \alpha$.

  By Lemma \ref{endtoend} we know that $\gamma$ and $\xi$ are Dehn twists of certain order of almost radial arcs. Using Remark \ref{almstgeod1} we can further say that there are almost radial end-to-end geodesic arcs $\gamma_0, \xi_0$ such that for some $m_1, m_2 \in \mathbb{Z}$
  \begin{equation}
  \gamma = \mathcal{D}^{m_1}_\alpha(\gamma_0), ~~ \xi = \mathcal{D}^{m_2}_\alpha(\xi_0) 	
  \end{equation}
  with $\gamma_0$ and $\xi_0$ are either identical or disjoint. The first monotonicity appears for $m_2 - m_1 >0$ and the second appears for $m_2 - m_1 <0$. Observe that $\gamma$ and $\mathcal{D}^{m_1}_\alpha(\xi_0)$ are either identical or disjoint. So the sign of $m_2 -m_1$, in some sense, measures the {\it amount} of Dehn twists applied to $\xi$ with respect to $\gamma$. Let us assume that $m_2 - m_1 >0$, the other case can be dealt with similarly.

  Recall that $x$ is the point of intersection between $\gamma$ and $\alpha$. Let $y$ be the point of intersection between $\xi$ and $\alpha$. Now consider the right half of $\mathcal{C}_\alpha \setminus \alpha$ and consider the points of intersection $x_1, x_2, ...$ between $\gamma$ and $\xi$ arranged in the ascending order of their distances from $x$ measured along $\gamma$. For each $x_i$ consider the subarc $\gamma_i$ of $\gamma$ between $x_i$ and $x$ and the subarc $\xi_i$ of $\xi$ between $x_i$ and $y$. Let $\gamma_i: [0, 1] \to \mathcal{C}_\alpha$ and $\xi_i: [0, 1] \to \mathcal{C}_\alpha$ be the parametrization of $\gamma_i$ and $\xi_i$ respectively such that $\gamma_i(0) = x_i = \xi_i(0)$ and $\gamma_i(1) = x, \xi_i(1) = y$. Using the definition of Dehn twist homeomorphism and length minimizing homotopy we may observe that there is a smooth homotopy $H_i:[0, 1] \times [0, 1] \to \mathcal{C}_\alpha$ between $\gamma_i$ and $\xi_i$ that has the following properties: $H_i(s, 0) = \gamma(s)$, $H_i(s, 1) = \xi_i(s)$, $H_i(0, t) = x_i$ and $ H_i(1, .)$ maps $[0, 1]$ to $\mathbb{S}^1 \simeq \alpha.$ Moreover the last map is orientation reversing with respect to the orientation of $\mathbb{S}^1 \simeq \alpha$. Lifting this homotopy to $\mathbb{H}^2$ we obtain lifts of $\xi_i$, $\gamma$ and $\alpha$ that forms a geodesic triangle $T_i$. Since $H_i(1, .): [0, 1] \to \mathbb{S}^1 \simeq \alpha$ is orientation reversing using the orientation of our fixed lift of $\alpha$ we conclude that the lift of $H_i(1, .)$ increases height. Making proper choices of these lifts now one has $T_i \subset T_{i+1}$. Hence $(1)$ follows from the area comparison of these two triangles via \eqref{areapolygon}.

  For the second part we need to consider end-to-end geodesic arcs $\xi$ in $\mathcal{C}_\alpha$ that intersects $\gamma$ large number of times. Using Lemma \ref{endtoend} we observe that it suffices to consider end-to-end geodesic arcs of the form $(\mathcal{D}^k_\alpha(\eta))$ where $\eta$ is an almost radial arc in $\mathcal{C}_\alpha$. We shall first study the geodesic laminations that are obtained as limits of end-to-end geodesic arcs of this last form. Let $L$ be one such geodesic lamination and consider the angle set $\Theta_S(\gamma, L) = (\phi_1, \phi_2, ....)$. The structure of $L$ is easy to describe. $L$ consists of three leaves one of which is $\alpha$. Rest of the two leaves spirals around $\alpha$ one in each component of $\mathcal{C}_\alpha \setminus \alpha$. Since $\alpha$ is the only minimal component of $L$ it follows that $\Theta_S(\gamma, L)$ has exactly one point of accumulation that corresponds to the point of intersection between $\gamma$ and $\alpha$ i.e. $\phi$. Hence we have
  \begin{equation*}
   |\{\phi_i: \phi_i \notin (\phi - \epsilon, \phi +\epsilon)\}| \le N_L(\epsilon)
  \end{equation*}
  for some integer $N_L(\epsilon)$ that a priori depends on $L$. To understand this dependence observe first that $L$ is determined by two things: $(i)$ the direction of spiraling of the two isolated leaves around $\alpha$ (see \S 1.4.2) and $(ii)$ the two end points of $L$ on the two components of $\partial{\mathcal{C}_\alpha}$. Since there are two possible directions in which the two isolated leaves of $L$ may spiral around $\alpha$ it suffices to take care of these two cases separately. 
  
  Let $\mathcal{L}^\pm$ be the collection of all those lamination (as above) whose isolated leaves respectively have $\pm$ direction of spiraling around $\alpha$. Hence any two laminations in $\mathcal{L}^+$ (or in $\mathcal{L}^-$) differ only by their end points on $\partial{\mathcal{C}_\alpha}$. Observe that by rotating each leaf appropriately (to match these two end points) any lamination in $\mathcal{L}^+$ (or in $\mathcal{L}^-$) can be obtained from any other. It is not difficult to observe that the effect of these rotations on $N_L(\epsilon)$ is continuous with respect to the angles of rotation. Hence $N_L(\epsilon)$ can be made independent of the lamination. Let us denote this uniform bound by $\mathcal{N}_\alpha(\epsilon).$ Finally for any lamination $L$ as above and $\Theta_S(\gamma, L) = (\phi_1, \phi_2, ....)$ we have
  \begin{equation}\label{count}
  |\{\phi_i: \phi_i \notin (\phi - \epsilon, \phi +\epsilon)\}| \le \mathcal{N}_\alpha(\epsilon).
  \end{equation}
  Now we are ready to prove the second part of our theorem. We argue by contradiction. So we have an $\epsilon >0$ for which we have almost radial arcs $\xi_j$ and a sequence of integers $(n_j)$ with $\Theta_S(\gamma, \mathcal{D}_\alpha^{n_j}(\xi_j)) = (\theta_1^j, \cdots, \theta_{m_j}^j)$ such that 
  \begin{equation}\label{lbound}
  m_j - \#|\{ i: \theta_i^j \in (\phi - \epsilon, \phi + \epsilon) \}| \to \infty.	
  \end{equation}
  Now extract a subsequence of $\xi_j$ that converges to an end-to-end geodesic arc $\xi_0$. Up to extracting subsequences, the limits of $\mathcal{D}_\alpha^{n_j}(\xi_j)$ and $\mathcal{D}_\alpha^{n_j}(\xi_0)$ are the same. Let $L_0$ be this limit and let $\Theta_S(\gamma, L_0) = (\phi_1^0, \phi_2^0, \cdots)$. The convergence $\mathcal{D}_\alpha^{n_j}(\xi_j) \to L_0$ via \eqref{lbound} implies that $$|\{\phi_i^0: \phi_i^0 \notin (\phi - \epsilon, \phi +\epsilon)\}| = \infty$$ which is contradictory to \eqref{count} via the convergence $\mathcal{D}_\alpha^{n_j}(\xi_0) \to L_0$ .  
  \end{proof}

  \subsubsection{The general case}
  Now let $\gamma$ and $\alpha_1, \cdots, \alpha_k$ are as in \S 3.2. Let $\eta$ be a simple closed geodesic which we would twist along different $\alpha_i$. Assume that $\eta$ intersects each $\alpha_i$. Recall that $T_{n_1, ..., n_k}(\eta)$ denotes the geodesic in the free homotopy class of $D_{\alpha_1}^{n_1} \circ \cdots \circ D_{\alpha_k}^{n_k}(\eta)$. By Remark \ref{eachintersect} and Lemma \ref{parallel} we know that for components $\gamma_j$ of $\gamma|_{\mathcal{C}_{\alpha_i}}$ and $\eta_j$ of $T_{n_1, \cdots, n_k}(\eta)|_{\mathcal{C}_{\alpha_i}}$ the number of intersections $\iota(\gamma_j, \eta_j) \approx |n_i|.$
 
 Now we divide $\gamma$ into different pieces $\gamma = \cup_{j=1}^l \gamma_j$ such that a $\gamma_j$ is contained in either one of the ${\mathcal{C}_{\alpha_i}}$ or in the complement of all these collars. Assume that $\gamma_j$ and $\gamma_{j+1}$ occur consecutively along $\gamma$. So
  \begin{equation*}
  	\Theta_S(\gamma, {T_{n_1, \cdots, n_k}}(\eta)) = (\Theta_S(\gamma_1, {T_{n_1, \cdots, n_k}}(\eta)), \cdots, \Theta_S(\gamma_l, {T_{n_1, \cdots, n_k}}(\eta))).
  \end{equation*}
  Now let $\Phi(\gamma, {T_{n_1, \cdots, n_k}}(\eta)) = \{\psi_1, \cdots, \psi_p\}$ where $\psi_i$s are distinct. Let $\epsilon_0$ be the minimum distance between any two distinct $\psi_i, \psi_j$. For each $i =1, 2, ..., k$ let $I_i$ be the collection of $j$ for which $\gamma_j \subset {\mathcal{C}_{\alpha_i}}$.
  \begin{thm}\label{generalangle}
  Assume that $n_i \to \infty$ for $i =  1, 2, .. k$. For $j \in I_i$ let $\phi_j$ be the angle of intersection between $\gamma_j$ and $\alpha_i$. 
  For any $\epsilon > 0$ let $\mathcal{P}_{\phi_j}^\epsilon(\gamma, T_{n_1, ..., n_k}(\eta))$ be the ordered subset of $\Theta_S(\gamma, {T_{n_1, ..., n_k}}(\eta))$ consisting of angles in $\Theta_S(\gamma_j, T_{n_1, ..., n_k}(\eta))$
  with magnitude in $(\phi_j-\epsilon, \phi_j + \epsilon)$. Then for any $\epsilon < \epsilon_0$ one has
  \begin{equation}\label{cardin2}
  \#|\mathcal{P}_{\phi_j}^\epsilon(\gamma, T_{n_1, \cdots, n_k}(\eta))| \approx \iota(\alpha_i, \eta){n_i}.
  \end{equation}
  On the other hand, for $j \notin \cup_{i=1}^k I_i$ the cardinality of $\Theta_S(\gamma_j, {T_{n_1, ..., n_k}}(\eta))$ is uniformly bounded independent of $(n_1, \ldots, n_k)$.  	
  \end{thm}
  \begin{proof}
  	The first part follows from Theorem \ref{asympt} and the second part follows from Lemma \ref{uniformintersection} in the Appendix.
  \end{proof}
  \begin{rems}\label{anglepatch}
  $(i)$ As in \S 1.4.2 consider a sequence $(n_{1, i}, n_{2, i}, ..., n_{k, i})$ such that ${T_{n_{1, i}, ..., n_{k, i}}}(\eta)$ converges to the geodesic lamination $L_{\sign(1)\alpha_1, ..., \sign(k)\alpha_k}(\eta)$ where $\sign(j)$ denotes the limiting sign of the sequence $(n_i)$. Observe that $\Theta_S(\gamma, \cup_{i=1}^k \alpha_i)$ is recognizable from the collection $(\Theta_S(\gamma, T_{n_{1, j}, ..., n_{k, j}}(\eta)))_j$. 

  $(ii)$ Fix $\gamma$, $\alpha_1, \cdots, \alpha_k$ and $\epsilon < \epsilon_0$ and consider the asymptotic in \eqref{cardin2}. A priori it depends on $\eta$. This dependence is uniform for $\eta \in \mathcal{N}(m_1, \cdots, m_k, l)$ (see Lemma \ref{uniformintersection} in the Appendix). To see this, by Theorem \ref{asympt}(2), it suffices to observe that all but finitely many points of intersections between $\gamma$ and ${T_{n_{1, i}, ..., n_{k, i}}}(\eta)$, in a uniform way, stays inside $\cup_{i=1}^k \mathcal{C}_{\alpha_i}$. This is the statement of Lemma \ref{uniformintersection} proved in the Appendix.
  \end{rems}
  We end this section with a description of $\Theta_S(\gamma, L)$ where $\gamma$ is an end-to-end arc in $\mathcal{C}_\alpha$ and $L$ is a geodesic lamination in $\mathcal{C}_\alpha$ that has exactly two leaves one of which is $\alpha$ and the other, $\ell$, starts at a point on $\partial {\mathcal{C}_\alpha}$ and spirals around $\alpha$ staying entirely in one of the components of ${\mathcal{C}_\alpha} \setminus \alpha$.  
  \begin{lem}\label{anglelamination}
  Let $\Theta_S(\gamma, \alpha) = (\phi)$ and $\Theta_S(\gamma, L) = (\theta_1, \cdots, \theta_n, \cdots)$. Then the sequence $(\theta_i)$ is strictly monotone and converges to $\phi$.
  \end{lem}
  \begin{figure}[H]
  	\includegraphics[scale=.35]{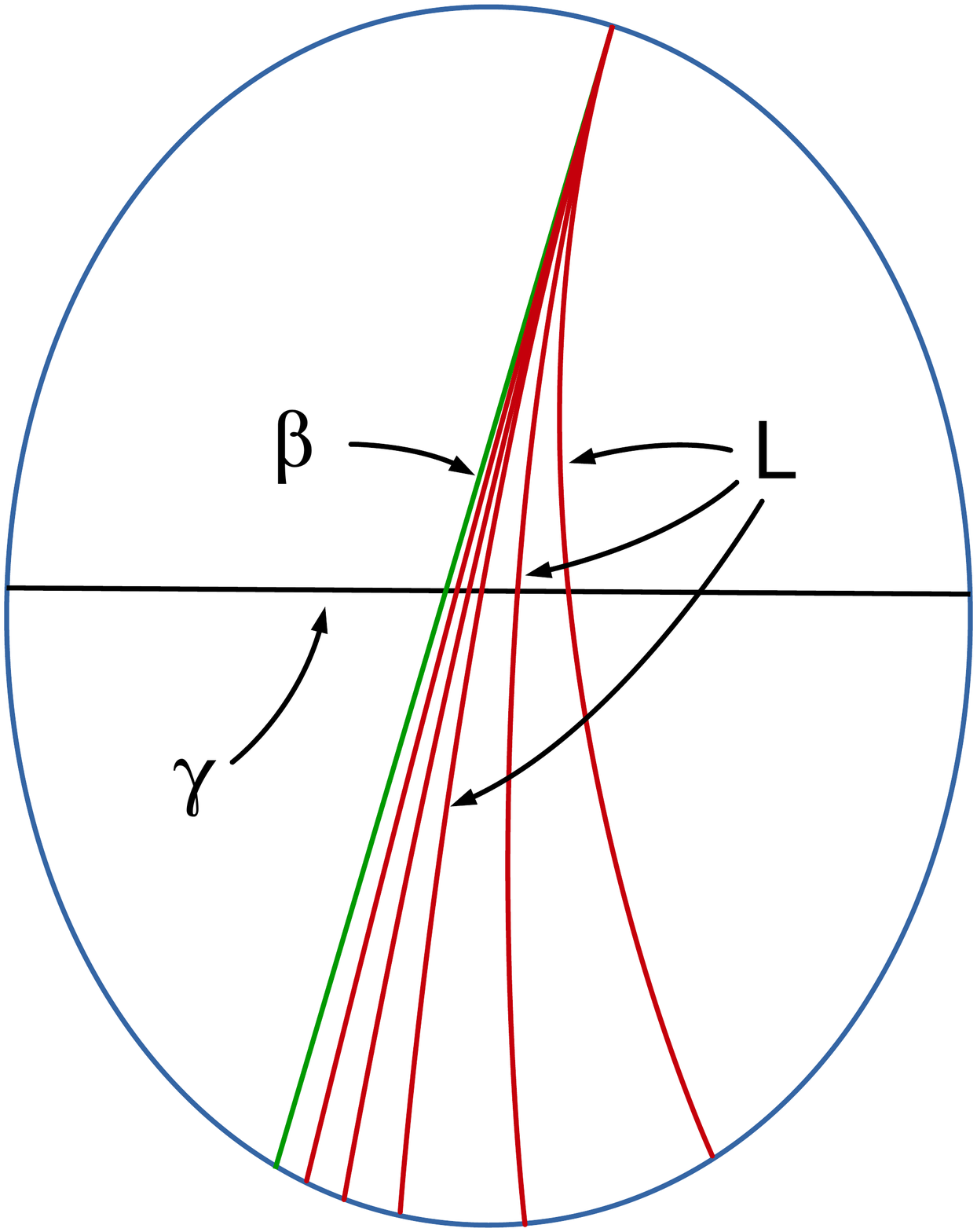}
  	\caption{Spiral}
  	\label{fig:spiral}
  \end{figure} 
  \begin{proof}
  Arguments here are very similar to those in the proof of Theorem \ref{asympt}. Recall that we have fixed one set of Fermi coordinates om $\mathcal{C}_\alpha$ and with respect these coordinates there is a precise direction in which $\ell$ spirals around $\alpha$. Assume that this direction is negative. The case of positive direction is very similar.
  	
  Taking one set of lifts of $\alpha, \gamma$ and $L$ our current situation looks like Figure \ref{fig:spiral}. To prove the monotonicity between $\theta_i$ and $\theta_j$ we compare the areas of the two triangles formed by the two lifts of $L$ corresponding to $\theta_i$ and $\theta_j$ with the fixed lifts of $\alpha, \gamma$. For the second part we use Theorem \ref{structure} to conclude that each point of accumulation of $\Theta_S(\gamma, L)$ corresponds to a point of intersection between $\gamma$ and a minimal sub-lamination of $L$. Since $L$ has exactly one minimal component, $\alpha$, the only limit of $(\theta_i)$ is $\phi$.
  \end{proof}
  \begin{rem}
   Using the description of $L_{\sign(1)\alpha_1, ..., \sign(k)\alpha_k}(\eta)$ from \S 1.4.2, the set up from Theorem \ref{generalangle} and the above lemma it follows that the ordered set of accumulation points of $\Theta_S(\gamma, L_{\sign(1)\alpha_1, ..., \sign(k)\alpha_k}(\eta))$ is precisely $\Theta_S(\gamma, \cup_{i=1}^k \alpha_i)$.
  \end{rem}

  \section{Proof of Theorem \ref{pants}}
  There are probably many different methods for constructing $\mathcal{P}_0$ once $\gamma_0$ is chosen. We describe one such method. We shall first choose $2g-2$ non-separating simple closed geodesics that make mutually disjoint set of angles with $\gamma_0$ and divides $S$ into $X$-pieces (four holed spheres with geodesic boundary).  Let us start by choosing one non-separating $\alpha_1 \in \mathcal{G}(S)$ that intersects $\gamma_0$ exactly once. Without loss of generality we may think that $\gamma_0$ and $\alpha_1$ are as in Figure \ref{fig:pantsconstr1}.
  \begin{figure}[H]
  	\includegraphics[scale=.35]{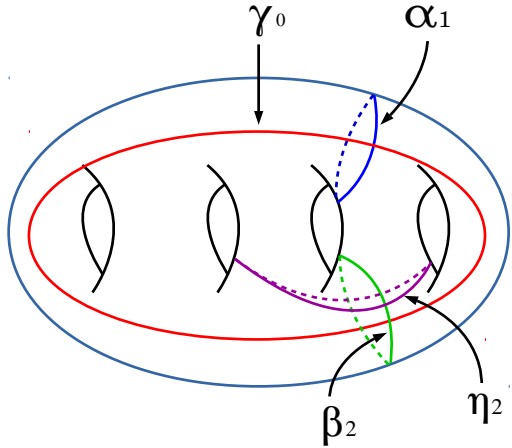}
  	\caption{Construction of $\mathcal{P}_0$}
  	\label{fig:pantsconstr1}
  \end{figure} 
  Now consider another simple closed geodesic $\beta_2$ as the green curve in Figure \ref{fig:pantsconstr1}. It also intersects $\gamma_0$ exactly once and do not intersect $\alpha_1$. If $\Phi(\gamma_0, \alpha_1) \neq \Phi(\gamma_0, \beta_2)$ then we choose $\alpha_2 =\beta_2$. If $\Phi(\gamma_0, \alpha_1) = \Phi(\gamma_0, \beta_2)$ then we modify $\beta_2$ as follows. Consider a simple closed geodesic $\eta_2$, as the purple curve in Figure \ref{fig:pantsconstr1}, that does not intersect $\alpha_1$and $\gamma_0$ but intersects $\beta_2$ exactly twice. Observe that $\mathcal{D}_{\eta_2}(\beta_2)$ intersects $\gamma_0$ exactly once. Moreover we have the following monotonicity. 
 \begin{claim}\label{pantsrotate}
  Let $\Theta(\gamma_0, \beta_2) = (\phi_1)$ and $\Theta(\gamma_0, \mathcal{D}_{\eta_2}(\beta_2)) = (\psi_1)$. Then $\phi_1 > \psi_1$.
  \end{claim}
 \begin{proof}
 As in the proof of Theorem \ref{asympt} we would lift $\gamma_0$, $\beta_2$ and $\mathcal{D}_{\eta_2}(\beta_2)$ to $\mathbb{H}^2$ and compare the angles there. For that we consider the point of intersection $y_0$ between $\gamma_0$ and $\beta_2$. In Figure \ref{fig:clearpic} the light green curve represents $\beta_2$, the magenta curve represents ${\eta_2}$, the violet curve represents $\mathcal{D}_{\eta_2}(\beta_2)$ and the red arc represents an arc of $\gamma_0$ corresponding to the angle $\phi_1$. Now fix one set of lifts of $\gamma_0$ and $\beta_2$ to $\mathbb{H}^2$ that intersect each other at a fixed point $y$. 
 
 Fix one set of Fermi coordinates on $\mathcal{C}_{\eta_2}$ and orient $\eta_2$ according to the orientation explained in \S 1.3.1 via these set of coordinates.  Observe that ${\eta_2}$ and $\beta_2$ intersect at two points and these two points divide $\beta_2$ into two geodesic arcs one of which contains $y_0$. Denote this last arc by $a_{\beta_2}$ and without loss of generality assume that this latter arc's restriction to $\mathcal{C}_{\eta_2}$ is contained in the left half of $\mathcal{C}_{\eta_2} \setminus \eta_2$. Observe that ${\eta_2}$ and $\mathcal{D}_{\eta_2}(\beta_2)$ also intersect at two points and these two points divide $\mathcal{D}_{\eta_2}(\beta_2)$ into two geodesic arcs one of which intersects $a_{\beta_2}$. Denote this arc by $a_{\mathcal{D}_{\eta_2}(\beta_2)}$. Consider parametrization $a_{\beta_2}: [0, 1] \to S$ and $a_{\mathcal{D}_{\eta_2}(\beta_2)}: [0, 1] \to S$ of these two arcs. Using the definitions of Dehn twist and length minimizing homotopy we observe that there is a smooth homotopy $H: [0, 1] \times [0, 1] \to S$ between $a_{\beta_2}$ and $a_{\mathcal{D}_{\eta_2}(\beta_2)}$ that has the following properties: $H(s, 0) = a_{\beta_2}(s)$, $H(s, 1) = a_{\mathcal{D}_{\eta_2}(\beta_2)}(s)$ and $H(0, t), H(1, t)$ maps $[0, 1]$ to $\eta_2$. Moreover these last two maps are orientation preserving with respect to the orientation of $\eta_2$. 
 
 \begin{figure}[H]
 	\includegraphics[scale=.35]{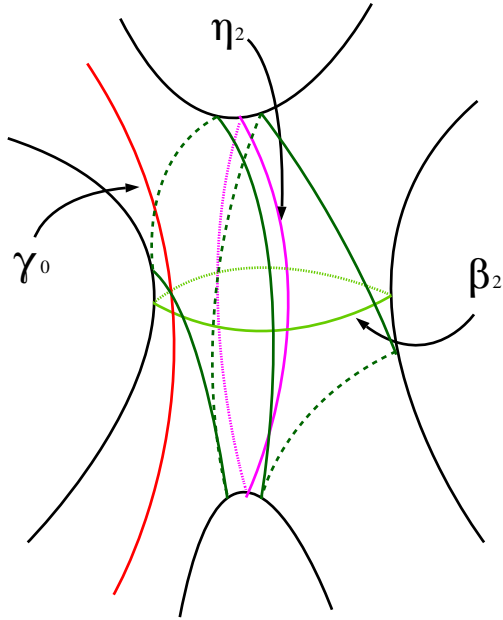}
 	\caption{For separable geodesics}
 	\label{fig:clearpic}
 \end{figure}
 
 To lift this homotopy to $\mathbb{H}^2$ we consider two lifts of $\eta_2$ as the two purple curves in Figure \ref{fig:lifted}. Observe that the above orientation of $\eta_2$ provides orientations of these two geodesics. This induced orientation increases height of the left lift and decreases height for the right lift. Thus lifting $H$ to $\mathbb{H}^2$ we obtain Figure \ref{fig:lifted}. Now it is evident that there is a point of intersection $x$ between the lifts $\tilde{a}_{\beta_2}$ of $a_{\beta_2}$ and $\tilde{a}_{\mathcal{D}_{\eta_2}(\beta_2)}$ of $a_{\mathcal{D}_{\eta_2}(\beta_2)}$. 
 
 We have two cases. First, $x$ and $y$ are identical. In this case our claim follows from the property that $H$ moves the two end points of $\tilde{a}_{\beta_2}$ along the two lifts of $\eta_2$ along the orientation $\eta_2$ (on $S$). In the second case $x$ and $y$ are distinct points. From Figure \ref{fig:lifted} we can assume that the homotopy $H$ between $\tilde{a}_{\beta_2}$ and $\tilde{a}_{\mathcal{D}_{\eta_2}(\beta_2)}$ is a rotation around $x$ sending $\tilde{a}_{\beta_2}$ to  $\tilde{a}_{\mathcal{D}_{\eta_2}(\beta_2)}$. 
 \begin{figure}[H]
 	\includegraphics[scale=.3]{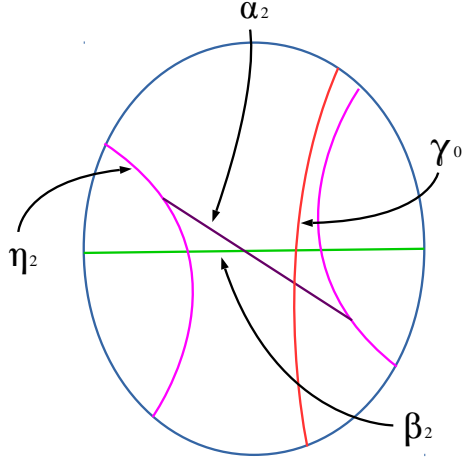}
 	\caption{Lifted on $\mathbb{H}^2$}
 	\label{fig:lifted}
 \end{figure}
Observe that $x$ divides $\tilde{a}_{\beta_2}$ into two connected components. Let $\tilde{b}_{\beta_2}$ be the closure of the component that contains $y$. So only the right lift of $\eta_2$ intersects $\tilde{b}_{\beta_2}$, say at $z$, and $H$ homotopes $\tilde{b}_{\beta_2}$ to a subarc $\tilde{b}_{\mathcal{D}_{\eta_2}(\beta_2)}$ of $\tilde{a}_{\mathcal{D}_{\eta_2}(\beta_2)}$. The monotonicity now follows from the positivity of the area of the triangle formed by $\tilde{b}_{\mathcal{D}_{\eta_2}(\beta_2)}$, $\tilde{b}_{\beta_2}$ and the image of $y$ under $H$ that is a subarc of the fixed lift of $\gamma_0$.
 \end{proof}
 \begin{figure}[H]
  	\includegraphics[scale=.25]{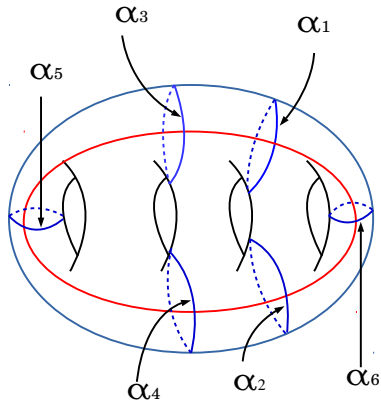}
  	\caption{All non-separating geodesics}
  	\label{fig:pansseparate}
 \end{figure}
 So we take $\alpha_2 = \mathcal{D}_\eta(\beta_2)$. We can repeat this procedure until we get a collection of non-separating simple closed geodesics $\alpha_1, ..., \alpha_{2g-2}$ that divide $S$ into a collection of $X$-pieces. Figure \ref{fig:pansseparate} explains this situation. In each of these $X$ pieces we have the situation as in Figure \ref{fig:xpiece} where the red arcs are the subarcs of $\gamma_0$. 
 \begin{figure}[H]
 	\includegraphics[scale=.25]{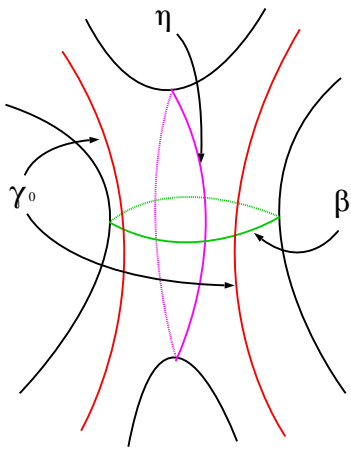}
 	\caption{Separating geodesics}
 	\label{fig:xpiece}
 \end{figure}
 Here we consider two simple closed geodesics $\beta$ and $\eta$ as in Figure \ref{fig:xpiece} where $\beta$ is separating and intersects $\gamma_0$ twice and $\eta$ is non-separating and do not intersect $\gamma_0$. Observe that this situation is very similar to the situation for non-separating geodesics above. The only difference in this situation is that now we have two subarcs of $\gamma_0$ instead of one. Arguments in the proof of Claim \ref{pantsrotate} work for each of these two arcs. Hence sufficient number of Dehn twist along $\eta$ would make sure that $\Phi(\gamma_0, \mathcal{D}^n_\eta(\beta))$ is disjoint from any finite collection of angles.  
\section{Proof of the Main Theorem}
 In this section we use the asymptotic growth of lengths and asymptotic structure of angle sets from \S3 along with the theory of geodesic laminations to prove Theorem \ref{geodesic-pants}. So we consider two closed hyperbolic surfaces $S, S'$ of genus $g$ with identical length-angle spectrum. 
 
 We begin by considering a simple closed non-separating geodesic $\gamma_0$ on $S$ and a pants decomposition $\mathcal{P}_0 = \{ \alpha_i: i =1, \cdots, \alpha_{3g-3}\}$ of $S$ provided by Theorem \ref{pants}. Then we fix a sequence $\bar{v}_n = (v_{1, n} , \cdots, v_{3g-3, n}) \in \mathbb{Z}^n_{>0}$ such that $\lim_{n \to \infty} v_{i, n} = \infty$ for each $i$ and
 \begin{equation}
  \lim_{n \to \infty} \frac{v_{i+1, n}}{v_{i, n}} = 0 ~~ \textrm{for each} ~~ i = 1, \cdots, 3g-2.
 \end{equation}
Then we consider the sequence of simple closed geodesics $T_{\bar{v}_n}(\gamma_0)$, in the free homotopy class of
\begin{equation*}
	D^{v_{1, n}}_{\alpha_1}  \circ D^{v_{2, n}}_{\alpha_2} \circ \cdots \circ  
	D^{v_{3g-3, n}}_{\alpha_{3g-3}}(\gamma_0).
\end{equation*}
As $n$ tends to infinity the sequence of closed geodesics $T_{\bar{v}_n}(\gamma_0)$ converge to the geodesic lamination $L_{\alpha_1, \cdots, \alpha_{3g-3}}(\gamma_0)$. Since $\mathcal{L}\Theta(S) = \mathcal{L}\Theta(S')$ we have simple closed geodesics $\gamma'_n, \delta'_n$ on $S'$ such that
 \begin{equation*}
 (\ell_S(\gamma_0), \ell_S(T_{\bar{v}_n}(\gamma_0)), \Theta_S(\gamma_0, T_{\bar{v}_n}(\gamma_0))) = (\ell_{S'}(\gamma'_n), \ell_{S'}(\delta'_n), \Theta_{S'}(\gamma'_n, \delta'_n).
 \end{equation*}
 A priori $\gamma'_n$ depends on $n$. Using the standard fact that the number of simple closed geodesics on any closed hyperbolic surface with length equal to (or bounded from above by) a fixed number is finite we can assume, up to extracting a subsequence, that we have a fixed simple closed geodesic $\gamma'_0$ such that
 \begin{equation*}
  (\ell_S(\gamma_0), \ell_S(T_{\bar{v}_n}(\gamma_0)), \Theta_S(\gamma_0, T_{\bar{v}_n}(\gamma_0)))= (\ell_{S'}(\gamma'_0), \ell_{S'}(\delta'_n), \Theta_{S'}(\gamma'_0, \delta'_n))
 \end{equation*}
 for some simple closed geodesics $\delta'_n$ on $S'$. Our goal now is to understand these simple closed geodesics $\delta'_n$. By Remark \ref{compact}, up to extracting a further subsequence, $\delta'_n$ converges to a geodesic lamination. We denote this geodesic lamination by $L$. Let $L_{\gamma'_0}$ be the smallest sub-lamination of $L$ that contains all those leaves of $L$ that intersect $\gamma'_0$. Hence $\Theta_{S'}(\gamma'_0, L) = \Theta_{S'}(\gamma'_0, L_{\gamma'_0})$. By Theorem \ref{structure} we have a finite collection of minimal sub-laminations $K_1,..., K_m$ whose complement in $L_{\gamma'_0}$ is a finite union of isolated leaves and each of these isolated leaves, along each of its ends, spirals around one of the $K_i$. Let $K_1, ..., K_l$ are those minimal sub-laminations of $L_{\gamma'_0}$ that intersect $\gamma'_0$. 
 \begin{lem}\label{accumulation:minimal}
 Each angle of intersection between $\gamma'_0$ and $K_i$ is a point of accumulation of $\Theta_{S'}(\gamma'_0, L_{\gamma'_0})$ and every point of accumulation of  $\Theta_{S'}(\gamma'_0, L_{\gamma'_0})$ is an angle of intersection between $\gamma'_0$ and one of the $K_i$s.	
 \end{lem}
 \begin{proof}
 Observe first that each $K_i$ is also a minimal component of $L$.  Since $L$ is the limit of a sequence of simple closed geodesics it follows that if $K_i$ is a simple closed geodesic then there is a leaf of $L$ that spirals around $K_i$. In particular this spiraling leaf must be in $L_{\gamma'_0}$. Now we have two possible type of minimal components: not isolated and isolated. If $K_i$ is not isolated then the lemma follows from the definition (of not isolated). If $K_i$ is isolated then it follows from the above observation that there is at least one isolated leaf of $L_{\gamma'_0}$ that spirals around $K_i$. For the reverse direction observe that a point of accumulation of $\Theta_{S'}(\gamma'_0, L_{\gamma'_0})$ can not correspond to an intersection between $\gamma'_0$ and an isolated leaf. Hence we have the lemma by Theorem \ref{structure}.	 
 \end{proof}
 Now we use our angle sets explicitly to have deeper understanding of these $K_i$s. It is probably believable that if one $K_i$ is not a simple closed geodesic then the angle sets $\Theta_S(\gamma'_0, \delta'_n)$ should look significantly different from $\Theta_S(\gamma'_0, T_{\bar{v}_n}(\gamma_0))$. For our purpose the next result would suffice.
 \begin{thm}\label{pants1}
 Each of $K_1, ..., K_l$ is a simple closed geodesic.
\end{thm}
\begin{proof}
 We begin by considering the angles in $$\Phi(\gamma_0, \cup_{i=1}^{3g-3} \alpha_i) = \{\Phi_1, \cdots, \Phi_M \}$$ where $\Phi_i$s are distinct. Let $\epsilon_0$ be the minimum of the distances between distinct $\Phi$s. For any $\epsilon < \epsilon_0/4$ the neighborhoods $I^\epsilon_i = (\Phi_i - \epsilon, \Phi_i + \epsilon)$ are at least $\epsilon$ distance apart. This makes sure that whenever we have an ordered subset $(\theta_1, \cdots, \theta_k)$ of some $\Theta(\gamma_0, \delta_{v_n})$ that has the property that $$|\theta_i - \theta_j|< \epsilon ~~ \textrm{for all}~~ i, j ~~ \textrm{and one of the} ~~ \theta_j \in I^\epsilon_l$$ then $\theta_i \in I^\epsilon_l$ for all $i=1, \cdots, k$. Now let $\phi$ be an angle of intersection between $\gamma'_0$ and one of the $K_i$s. We shall first show that $\phi = \Phi_k$ for some $k$. 
 
 By the last lemma $\phi$ is a point of accumulation of $\Theta_{S'}(\gamma'_0, L_{\gamma'_0})$ and so we have an ordered sequence $(\phi_1, \phi_2, \cdots)$ of $\Theta_{S'}(\gamma'_0, L_{\gamma'_0})$ that converges to $\phi$. We may choose this sequence in such a way that $\phi_j \in (\phi - \frac{\epsilon}{4}, \phi + \frac{\epsilon}{4})$ for all $j$. Using the convergence $\delta'_n \to L$ we have ordered subset $(\psi_1^n, \cdots, \psi_{k_n}^n)$ of $\Theta(\gamma'_0, \delta'_n)$,  for $n$ sufficiently large, such that $\psi_i^n \to \phi_i$ as $n \to \infty$ (in particular the size of these ordered sets $k_n \to \infty$). We may further assume by making $n$ larger if necessary that each $\psi_i^n \in (\phi_i - \frac{\epsilon}{2}, \phi_i + \frac{\epsilon}{2})$. This implies that $$|\psi_i^n - \psi_j^n|< \epsilon ~~ \textrm{for each} ~~ i, j = 1, \cdots, k_n.$$ 
 
 Using $\Theta_{S'}(\gamma'_0, \delta'_n) = \Theta_S(\gamma_0, T_{\bar{v}_n}(\gamma_0))$ and the first paragraph of this proof we conclude that each $\psi_i^n \in I^\epsilon_k$ for some $k$ independent of $i$. Since we have finitely many $\Phi_k$s, up to extracting a subsequence, we may assume that each $\psi_i^n \in I^\epsilon_k$ for some fixed $k$ independent of $i$ and $n$. Now let $\Phi_k$ be an angle of intersection between $\gamma_0$ and $\alpha_l$. By Theorem \ref{generalangle} it follows that if we choose $\epsilon$ sufficiently small then we can make sure, up to discarding a few angles if necessary, that $(\psi_1^n, \cdots, \psi_{k_n}^n)$ is an ordered subset of  $\Theta(\gamma_0|_{\mathcal{C}_{\alpha_l}}, T_{\bar{v}_n}(\gamma_0)|_{\mathcal{C}_{\alpha_l}})$. Depending on if $\alpha_l$ is separating or not we have two cases. If $\alpha_l$ is non-separating then $\Phi_k$ is the only angle of intersection between $\alpha_l$ and $\gamma_0$. If $\alpha_l$ is separating then we have two points of intersection between $\alpha_l$ and $\gamma_0$. If both the angles at these two intersections are equal to $\Phi_k$ then again we are okay. The last situation is that the two angles are different and one of them is $\Phi_k$. Let $\gamma_0^1$ denote $\gamma_0|_{\mathcal{C}_{\alpha_l}}$ in the first two cases and the subarc of $\gamma_0|_{\mathcal{C}_{\alpha_l}}$ that corresponds to the angle $\Phi_k$ in the second case. Hence $(\psi_1^n, \cdots, \psi_{k_n}^n)$ is an ordered subset of  $\Theta(\gamma_0^1, T_{\bar{v}_n}(\gamma_0)|_{\mathcal{C}_{\alpha_l}})$.
 Using the convergence $T_{\bar{v}_n}(\gamma_0) \to L_{\alpha_1, \cdots, \alpha_{3g-3}}(\gamma_0)$ and our assumption $\psi_i^n \to \phi_i$ we conclude that $\phi_i \in \Theta(\gamma_0^1, L_{\alpha_1, \cdots, \alpha_{3g-3}}(\gamma_0)|_{\mathcal{C}_{\alpha_l}})$. 
 
 From Lemma \ref{anglelamination} we have a description for $\Theta(\gamma_0^1, L_{\alpha_1, \cdots, \alpha_{3g-3}}(\gamma_0)|_{\mathcal{C}_{\alpha_l}}).$ In particular, a fixed angle can appear in $\Theta(\gamma_0^1, L_{\alpha_1, \cdots, \alpha_{3g-3}}(\gamma_0)|_{\mathcal{C}_{\alpha_l}})$ at most $2\iota(\gamma_0^1, \alpha_l)$ times and $\Phi_k$ is its only accumulation point. Let $\psi$ be an angle in $\Theta(\gamma_0^1, L_{\alpha_1, \cdots, \alpha_{3g-3}}(\gamma_0)|_{\mathcal{C}_{\alpha_l}})$ that is not equal to $\Phi_k$. Choose $\epsilon > 0$ small enough such that $(\psi - \epsilon, \psi + \epsilon)$ and $I^\epsilon_1, \cdots, I^\epsilon_k$ are disjoint and \eqref{cardin2} is true. Then the number of angles in any $\Theta_S(\gamma_0^1, T_{\bar{v}_n}(\gamma_0))$ that lie in $(\psi - \epsilon, \psi + \epsilon)$ is uniformly bounded, independent of $n$. In particular, only finitely many $\phi_i$ can be equal to $\psi$. Hence $\phi = \Phi_k$. 
 
Now we show that each $K_i$ contains an isolated leaf. We argue by contradiction and assume that $K_i$ does not contain any isolated leaf. Hence each point in $\gamma'_0 \cap K_i$ is a point of accumulation of $\gamma'_0 \cap K_i$. In particular $\gamma'_0 \cap K_i$ contains uncountably many points. By the first part of this proof all the angles of these intersections must come from the finite set  $\{\Phi_1, \cdots, \Phi_k \}$. Let $\ell$ be a leaf of $K_i$. Since $K_i$ is minimal $\ell$ must intersect $\gamma'_0$ infinitely many times. Let $\ell_0$ be a subarc of $\ell$ between two such intersections. Using minimality once again we find subarcs $\ell_i$ of possibly other leaves of $K_i$ such that $\ell_i \to \ell_0$ uniformly. Now lift the whole situation on $\mathbb{H}^2$. Let $\gamma_1, \gamma_2$ be two fixed lifts of $\gamma'_0$ such that a lift $\tilde{\ell}_0$ of $\ell_0$ joins $\gamma_1$ and $\gamma_2$. Using the fact that $\ell_i \to \ell$ uniformly we can find lifts $\tilde{\ell}_i$ of each $\ell_i$ such that $\tilde{\ell}_i$ also joins $\gamma_1$ and $\gamma_2$. Since $\ell_0$ and $\ell_i$s are parts of a geodesic lamination their lifts $\tilde{\ell}_0$ and $\tilde{\ell}_i$ are mutually disjoint. Thus subarcs of $\gamma_1, \gamma_2, \tilde{\ell}_i$ and $\tilde{\ell}_j$ for each $i \neq j$ bound a geodesic rectangle, say $R(i, j)$. 

Now consider the angles of intersections $\Theta(\gamma_1, \tilde{\ell}_i)$. Since they must be one of $\Phi_1, \cdots, \Phi_k$ we can extract a subsequence of $\ell_i$s such that $\Theta(\gamma_1, \tilde{\ell}_i) = (\Phi_l)$ for some $l$ independent of $i$. Extracting a further subsequence we can further ensure that $\Theta(\gamma_2, \tilde{\ell}_i) = (\Phi_k)$ for some $k$ independent of $i$. Finally we reach our contradiction by computing the area of $R(i, j)$ for this extracted sequence of $\ell_i, \ell_j$s (which, by our assumption, is equal to zero!)
\end{proof}  
Let us denote the simple closed geodesic $K_i$ by $\alpha'_i$. So there are leaves of $L_{\gamma'_0}$ that spiral around $\alpha'_1, ..., \alpha'_l$. Of course $L_{\gamma'_0}$ can have much complicated behavior {\it away} from $\gamma'_0$. Now we make this observation precise. 
\begin{defn}
	Let $p$ be a point of intersection between $\gamma'_0$ and $L_{\gamma'_0}$. Let $\ell$ be the leaf of $L_{\gamma'_0}$ that intersects $\gamma'_0$ at $p$. We say that $p$ corresponds to a spiraling if one of the half-leaves of $\ell$, determined by $p$, spiral around one of $\alpha'_1, \cdots, \alpha'_l$ along one of its (two) ends. 
\end{defn}
\begin{lem}\label{intersectionbound1}
	All but finitely many points of intersections between $\gamma'_0$ and $L_{\gamma'_0}$ correspond to a spiraling.
\end{lem}
\begin{proof}
	We argue by contradiction and assume that there are infinitely many points $x_1, \cdots, x_n, \cdots$ of intersections between $\gamma'_0$ and $L_{\gamma'_0}$ that does not correspond to a spiraling. Since $\gamma'_0$ is a closed geodesic the sequence of points $x_1, \cdots, x_n, \cdots$ have a point of accumulation on $\gamma'_0$. By Lemma \ref{accumulation:minimal} this point of accumulation must be a point of intersection between $\gamma'_0$ and a minimal component of $L_{\gamma'_0}$. By the last theorem it must be one of the $\alpha'_i$s. Since $\alpha'_i$ is a closed geodesic all but finitely many of $x_1, \cdots, x_n, \cdots$ corresponds to spiraling around $\alpha'_i$. This is a contradiction. 
\end{proof}
Let $N_{\gamma'_0}$ be the number of points of intersection between $\gamma'_0$ and $L_{\gamma'_0}$ that does not corresponds to a spiraling. By the last lemma and Lemma \ref{anglelamination} the (ordered) set of accumulation points of $\Theta_{S'}(\gamma'_0, L_{\gamma'_0})$ is exactly $\Theta_{S'}(\gamma'_0, \cup_{i=1}^k \alpha'_i)$. We now compare this with $\Theta_S(\gamma_0, \mathcal{P}_0)$
\begin{lem}\label{accumulation}
 As ordered sets $\Theta_S(\gamma_0, \mathcal{P}_0) = \Theta_{S'}(\gamma'_0, \cup_{i=1}^k \alpha'_i)$.
\end{lem}
\begin{proof}
 Let $\Theta_{S'}(\gamma'_0, \cup_{i=1}^k \alpha'_i) = (\theta_1, \ldots, \theta_p)$. By the last lemma all but $N_{\gamma'_0}$ angles in $\Theta(\gamma'_0, L_{\gamma'_0})$ corresponds to spiraling around one of the $\alpha'_1, ..., \alpha'_l$. Since $\Theta(\gamma'_0, L) = \Theta(\gamma'_0, L_{\gamma'_0})$ using the convergence $\delta'_n \to L$ and Lemma \ref{anglelamination} we conclude that $\Theta_{S'}(\gamma'_0, \delta'_n)$ consists of an ordered set $(I(\theta_1),  \ldots , I(\theta_p))$ where each entry in $I(\theta_i)$ lie in $(\theta_i - \epsilon, \theta_i + \epsilon)$ and rest of the angles in $\Theta_{S'}(\gamma'_0, \delta'_n)$ has cardinality bounded independent of $n$. 
 
 Since $\Theta_{S'}(\gamma'_0, \delta'_n) = \Theta_S(\gamma_0, T_{\bar{v}_n})$ we actually know that the latter consists of an ordered set $(I(\phi_1),  \ldots, I(\phi_l))$ where $I(\phi_i)$ looks like Theorem \ref{generalangle} and the rest of the angles in $\Theta_S(\gamma_0, \delta_{v_n})$ has uniformly bounded cardinality independent of $n$. Comparing the two descriptions of $\Theta_{S'}(\gamma'_0, \delta'_n) = \Theta_S(\gamma_0, T_{\bar{v}_n})$ we conclude the lemma.
 \end{proof}
 Recall that $L$ is the limit of a sequence of simple closed geodesics and there are only finitely many isolated leaves (in any geodesic lamination; Theorem \ref{structure}) in $L$ that spiral around $\alpha'_i$. Hence for each $i$ there are a finite and equal number of leaves spiraling around $\alpha'_i$ from both sides in the same direction. Let $\xi_i$ be the number of leaves that spiral around $\alpha'_i$ from one side. Since $L$ is the limit of $(\delta'_n)$ it follows that $\iota(\alpha'_i, \delta'_n) = \xi_i$.
\subsection*{Untwisting}
 From the above observations it is reasonable to think that, up to extracting a further subsequence, $(\delta'_n)$ are the images of a fixed simple closed geodesic $\delta'_0 \in \mathcal{G}(S')$ under various combinations of Dehn twists along $\alpha_i$s. We make this precise in the next proposition.
\begin{prop}\label{fixed}
 There is a subsequence $\delta'_{m_n}$ of $\delta'_n$ and a simple closed geodesic $\delta'_0$ such that
  \begin{equation}
   \delta'_{m_n} = \prod_{i=1}^{l} \mathcal{D}_{\alpha'_i}^{s^i_n}(\delta'_0)
  \end{equation}
 where $\iota(\gamma', \delta'_0) \le N_{\gamma'_0} + \sum_{i=1}^l {\xi_i}\cdot {\iota(\gamma', \alpha'_i)}$ and $N_{\gamma'_0}$ is the number from Lemma \ref{intersectionbound1}.
\end{prop}
\begin{proof}
 Let us start by recalling that a sub-lamination of $L$ spirals around $\alpha'_1, ..., \alpha'_l$. Since $\delta'_n \to L$ it follows that for $n$ sufficiently large $\delta'_n$ has {\it large number of twists} around each $\alpha'_i$. Hence applying Dehn twists to $\delta'_n$ along $\alpha'_i$s we can get simple closed geodesics that intersect $\gamma'_0$ fewer of times than $\delta'_n$ does. Following our notations from \S3 for $(s_1, \cdots, s_k) \in \mathbb{Z}^k$ let $T_{s_1, \cdots, s_k}(\delta'_n)$ denote the simple closed geodesic freely homotopic to $D_{\alpha'_1}^{s_1} \circ \cdots \circ D_{\alpha'_k}^{s_k}(\delta'_n)$. Consider a simple closed geodesic $\beta_n$ such that
 \begin{equation}
 	\iota(\gamma'_0, \beta_n) = \min_{(s_1, \cdots, s_k) \in \mathbb{Z}^k} \iota(\gamma'_0, T_{s_1, \cdots, s_k}(\delta'_n)).
 \end{equation} 
 A fairly straight forward topological argument provides that
 \begin{equation}\label{intrscbound}
 \iota(\gamma'_0, \beta_n) \le N_{\gamma'_0} + \sum_{i=1}^l {\xi_i}\cdot {\iota(\gamma'_0, \alpha'_i)}.
 \end{equation}
Let $\delta'_n =  T_{s_1^n, \cdots, s_k^n}(\beta_n)$. Now estimate the number $l$.

 \begin{lem}\label{pantsrecovered}
 The collection $\{ \alpha'_i: i =1, 2, ..., l \}$ forms a pants decomposition $\mathcal{P}'$ of $S'$. After a rearrangement of the indices
 \begin{equation*}
  \Theta(\gamma_0, \alpha_i)= \Theta(\gamma'_0, \alpha'_i).
 \end{equation*}
\end{lem}
\begin{proof}
 The angle set $\Theta(\gamma'_0, \cup_{i=1}^l \alpha'_i)$ is the set of accumulation points of $\Theta(\gamma', L_{\gamma'})$ and by Lemma \ref{accumulation} we have
 \begin{equation*}
 \Theta(\gamma'_0, \cup_{i=1}^l \alpha'_i) = \Theta(\gamma_0, \mathcal{P}_0).
 \end{equation*}
 Now fix one $i$ and for $\alpha_i$ consider a $\alpha'_i$ for which $\Phi(\gamma'_0, \alpha'_i) \cap \Phi(\gamma_0, \alpha_i) \neq \emptyset$. From \S4 and the last equality of angle sets we known that there are at most two choices for this.

 Consider $\phi \in \Phi(\gamma'_0, \alpha'_i) \cap \Phi(\gamma_0, \alpha_i)$. On $S'$ let $\phi$ be the angle of intersection between the subarc $\gamma'_i$ of $\gamma'_0|_{\mathcal{C}_{\alpha'_i}}$ and $\alpha'_i$. Recall the ordered subset $\mathcal{P}^\epsilon_\phi(\gamma'_0, \delta'_n)$ of $\Theta(\gamma'_0, \delta'_n)$ that consists of angles in $\Theta(\gamma'_i, \delta'_n)$ with magnitude in $(\phi-\epsilon, \phi + \epsilon)$. By Theorem \ref{generalangle} and Remark \ref{anglepatch} we have the asymptotic 
 \begin{equation*}
  \#|\mathcal{P}^\epsilon_\phi(\Theta(\gamma'_0, \delta'_n))| \approx  s^i_n \cdot \xi_i.
 \end{equation*}
 From our construction we know on the other hand that
 \begin{equation*}
  \#|\mathcal{P}^\epsilon_\phi(\gamma_0, T_{\bar{v}_n}(\gamma_0))| \approx \iota(\gamma_0, \alpha_i) \cdot v_{i, n}.
 \end{equation*}
 By Lemma \ref{intersectionbound1} the last two asymptotic counts are comparable i.e.
 \begin{equation}\label{twistnumb}
  \iota(\gamma_0, \alpha_i)\cdot v_{i, n} \approx {\xi_i}\cdot{s^i_n}.
 \end{equation}
 At this point we recall our choice:
 \begin{equation}\label{ratiozero}
  \lim_{n \to \infty} \frac{v_{i+1, n}}{v_{i, n}} = 0 ~~ \textrm{for all} ~~ i= 1, \cdots, 3g-3.
 \end{equation}
 By Lemma \ref{accumulation} we know that for any $\phi_1 \in \Theta_S(\gamma_0, \mathcal{P}_0)$ there is an $\alpha'_i$ such that $\phi_1 \in \Theta_{S'}(\gamma'_0, \alpha'_i)$. To estimate the number $l$ we first count how many angles in $\Theta_{S'}(\gamma'_0, \cup_{i=1}^l \alpha'_i)$ can belong to the same $\Theta_{S'}(\gamma'_0, \alpha'_i)$. If $\phi_1, \phi_2 \in \Theta_{S'}(\gamma'_0, \alpha'_i)$ then by \eqref{twistnumb} and \eqref{ratiozero} it follows that there is exactly one $j$ such that $\Theta_S(\gamma_0, \alpha_j) = (\phi_1, \phi_2)$. In particular, from the special properties of $\mathcal{P}_0$ from Theorem \ref{pants}, it follows that at most two angles in $\Theta_{S'}(\gamma'_0, \cup_{i=1}^l \alpha'_i)$ can belong to the same $\Theta_{S'}(\gamma'_0, \alpha'_i)$ and that happens only if they belong to one of the $\Theta_S(\gamma_0, \alpha_j)$. Hence $l$ is at least $3g-3$. Since $\alpha'_i$s are mutually disjoint, this number is the maximal possible. Therefore every $\alpha_i$ corresponds to a unique $\alpha'_i$ such that $\Theta_S(\gamma_0, \alpha_i) = \Theta_{S'}(\gamma'_0, \alpha'_i)$ and we have a pants decomposition $\mathcal{P}'_0 = \{ \alpha'_i: i =1, 2, ..., 3g-3 \}$ of $S'$.
\end{proof}
 Now we are ready to finish the proof of our proposition. It suffices to show that $\ell(\beta_n)$ is uniformly bounded. We argue by contradiction and assume that $\ell(\beta_n)$ is unbounded. In particular, there is at least one pair of pants $P$ determined by $\mathcal{P}'_0$ such that the length of $\beta_n$ restricted to $P$ do not stay bounded. Now recall that for each $i$ we have $\iota(\beta_n, \alpha'_i) = \xi_i$, a fixed finite number determined by $L$. Hence one of the subarcs of $\beta_n$ stays entirely inside $P$ whose length does not stay bounded. This implies that this subarc twists around one of the $\alpha'_i$ in $\partial{P}$ a large number of times. In particular, $\iota(\gamma'_0, \beta_n)$ does not stay bounded. This is a contradiction to \eqref{intrscbound}.
\end{proof}
 The only part of Theorem \ref{geodesic-pants} that remains to be proven is the following.
\begin{lem}
 After rearranging the indices according to Lemma \ref{pantsrecovered} for each $i=1, ..., 3g-3$ we have $\ell(\alpha'_i) = \ell(\alpha_i)$.
\end{lem}
\begin{proof}
 Recall that our geodesic $T_{\bar{v}_n}(\gamma_0)$ is the geodesic freely homotopic to $D_{\alpha_1}^{v_{1, n}} \circ \cdots D_{\alpha_{3g-3}}^{v_{3g-3, n}}(\gamma_0).$ Thus we have the following length comparison from Theorem \ref{length}
\begin{equation}\label{lengthcomparison1}
  \sum_{i=1}^{3g-3} \iota(\gamma_0, \alpha_i)\cdot(v_{i, n}-k_i)\cdot\ell(\alpha_i) \le \ell(T_{\bar{v}_n}(\gamma_0)) $$$$\le \ell(\gamma_0) + \sum_{i=1}^{3g-3} \iota(\gamma_0, \alpha_i)\cdot{v_{i, n}}\cdot\ell(\alpha_i)
\end{equation}
where $k_i$ are some fixed positive integers depending on $\alpha_1, ..., \alpha_{3g-3}, \gamma_0$. By the last proposition we also know that $\delta'_{m_n}$ is the geodesic freely homotopic to $\mathcal{D}_{\alpha'_1}^{s^1_n} \circ \cdots \circ\mathcal{D}_{\alpha'_{3g-3}}^{s^{3g-3}_n}(\delta'_0)$ which provides via Theorem \ref{length}
\begin{equation}\label{lengthcomparison2}
\sum_{i=1}^{3g-3} {\xi_i}\cdot(s^i_n - k'_i)\cdot \ell(\alpha'_i) \le \ell(\delta'_{m_n}) \le \ell(\delta'_0) + \sum_{i=1}^{3g-3} {\xi_i}\cdot{s^i_m}\cdot\ell(\alpha'_i)
\end{equation}
where $k'_i$ are fixed some positive integers depending on $\alpha'_1, ..., \alpha'_{3g-3}, \delta'_0$. The rest of the arguments consist of computing some limits using: $(1)$ the equality $\ell(\delta_{v_n}) = \ell(\delta'_n)$, $(2)$ the inequalities \eqref{lengthcomparison1} and \eqref{lengthcomparison2} and $(3)$ the asymptotic behavior \eqref{twistnumb}. For example, to prove $\ell(\alpha'_1) = \ell(\alpha_1)$ we use \eqref{twistnumb} to find that  \begin{equation*}
  \lim_{n \to \infty} \frac{\ell(\delta_{v_{m_n}})}{\iota(\gamma_0, \alpha_1)\cdot{v_{1, m_n}}} = \lim_{n \to \infty} \frac{\ell(\delta'_{m_n})}{{\xi_1}\cdot{s^1_{m_n}}}.
 \end{equation*}
Using \eqref{ratiozero} we observe that the left limit is $\ell(\alpha_1)$ by \eqref{lengthcomparison1} and the right limit is $\ell(\alpha'_1)$ by \eqref{lengthcomparison2}. Now we use induction and assume $\ell(\alpha_i) = \ell(\alpha'_i)$ for $i \le k-1$. Using \eqref{twistnumb} once again we obtain the equality 
\begin{equation*}
\lim_{n \to \infty} \frac{\ell(\delta_{v_{m_n}}) - \sum_{i=1}^{k-1}\iota(\gamma
	_0, \alpha_i)\cdot{v_{i, m_n}}\ell(\alpha_i)}{\iota(\gamma_0, \alpha_k)\cdot {v_{k, m_n}}} = \lim_{n \to \infty} \frac{\ell(\delta'_{m_n}) - \sum_{i=1}^{k-1}{\xi_i} \cdot {s^1_{m_n}}\ell(\alpha'_i)}{{\xi_k}\cdot{s^k_{m_n}}}.	
\end{equation*}
As above, using \eqref{ratiozero} we can observe that by \eqref{lengthcomparison1} the left limit is $\ell(\alpha_k)$ and by \eqref{lengthcomparison2} the right limit is $\ell(\alpha'_k)$.
 \end{proof}
 \section{Appendix}
 In this small section we explain some basic results used in the paper that are probably know to experts. 
 \subsection{Lengths of end-to-end arcs}
 The first result is about the length of an end-to-end arc inside the collar $\mathcal{C}_\alpha$ around $\alpha$.
 \begin{lem}\label{length1}
 	Let $\gamma$ be an end-to-end geodesic arc inside the collar $\mathcal{C}_\alpha$ around $\alpha$. Let $\eta$ be an almost radial arc in $\mathcal{C}_\alpha$. Then
 	\begin{equation*}
 		\ell(\gamma) \ge (\iota(\gamma, \eta) - 2)\ell(\alpha).
 	\end{equation*}
 \end{lem}
 \begin{proof}
 Let $\eta_\phi$ be the radial arc that does not intersect $\eta$. By Lemma \ref{parallel} we obtain that
 \begin{equation}\label{length2}
 	\iota(\gamma, \eta_\phi) \ge \iota(\gamma, \eta) -1.
 \end{equation}
 Consider a subarc $\gamma_i$ of $\gamma$ between two consecutive intersections with $\eta_\phi$. The projection from this subarc to $\alpha$ via the $\theta$ coordinate (of Fermi coordinates) is surjective. Since this projection is a length decreasing: $\ell(\gamma_i) \ge \ell(\alpha)$. We obtain the lemma by adding up all the pieces of $\gamma$ between different points of intersection with $\eta_\phi$ and \eqref{length2}.
 \end{proof}
 \subsection{Uniform bound on the number of intersections}
 Our next result is used in \S 3 where we study how the length and angle sets evolve under various Dehn twists. Let $\alpha_1, ..., \alpha_k$ be $k$ mutually disjoint simple closed geodesics. Let $\gamma$ be a simple closed geodesic that intersects each $\alpha_i$. We consider the sets $$\mathcal{N}(m_1, ..., m_k, l) = \{\eta \in \mathcal{GL}(S): \iota(\eta, \alpha_i) = m_i ~~ \textrm{for} ~~ i=1, 2, ..., k ~~$$$$ \textrm{and} ~~ \iota(\eta, \gamma) \le l \}.$$ 
 Our purpose here is to understand how the intersections between $\gamma$ and various Dehn twists of $\eta$ along $\alpha_1, \cdots, \alpha_k$ are located on $S$.
 \begin{lem}\label{uniformintersection}
 Then there is a $N= N(m_1, m_2, ..., m_k, l)$ such that for any $\eta \in \mathcal{N}(m_1, ..., m_k, l)$ and any tuple $(n_1, n_2, ..., n_k) \in \mathbb{Z}^k$ one has
 \begin{equation*}
 \iota(\gamma|_{S \setminus \cup_{i=1}^k \mathcal{C}_{\alpha_i}}, T_{n_1, n_2, ..., n_k}(\eta)) \le N.
 \end{equation*}
 \end{lem}
 \begin{proof}
 	Let $\eta \in \mathcal{N}(m_1, ..., m_k, l)$. Let $(n_{1, j}, n_{2, j}, ..., n_{k, j})$ $\in \mathbb{Z}^k$ be a sequence such that $ T_{n_{1, j}, n_{2, j}, ..., n_{k, j}}(\eta) \to L$ a geodesic lamination. The structure of $L$ is easy to describe. The minimal sub-laminations of $L$ are those of $\eta$ and $\alpha_1, \cdots, \alpha_k$. Hence we can find a $N = N(L)$ such that $\iota(\gamma|_{S \setminus \cup_{i=1}^k \mathcal{C}_{\alpha_i}}, L) \le N$.
 	
 	To argue by contradiction we assume that we have $\eta_j \in \mathcal{N}(m_1,\cdots, m_k, l)$ and a sequence  $(p_{1, j}, p_{2, j}, \cdots, p_{k, j})$ $\in \mathbb{Z}^k$ be such that for all $j$ 
 	\begin{equation}\label{lower}
 		\iota(\gamma|_{S \setminus \cup_{i=1}^k \mathcal{C}_{\alpha_i}}, T_{p_{1, j}, p_{2, j}, \cdots, p_{k, j}}(\eta_j)) \ge j.
 	\end{equation}
  Up to extracting a subsequence both $(\eta_j)$ and $(T_{p_{1, j}, p_{2, j}, \cdots, p_{k, j}}(\eta_j))$ converge to some geodesic laminations. It is not that difficult to see that if $\eta_j \to \eta$ then the limit of $T_{p_{1, j}, p_{2, j}, \cdots, p_{k, j}}(\eta_j)$ up to extracting correct subsequences is the same as the limit of $T_{p_{1, j}, p_{2, j}, \cdots, p_{k, j}}(\eta)$ which we denote by $L_0$. Now $\eta \in \mathcal{N}(m_1, \cdots, m_k, l)$ so by the first paragraph of this proof we have a $N_0 = N_0(L_0)$ such that
 \begin{equation*}
 	\iota(\gamma|_{S \setminus \cup_{i=1}^k \mathcal{C}_{\alpha_i}}, L_0) \le N_0.
 \end{equation*}
 On the other hand, from \eqref{lower} and the convergence $T_{p_{1, j}, p_{2, j}, \cdots, p_{k, j}}(\eta_j) \to L$, we have 
 \begin{equation*}
 	\iota(\gamma|_{S \setminus \cup_{i=1}^k \mathcal{C}_{\alpha_i}}, L_0) = \infty.
 \end{equation*}
Hence we have a contradiction.
\end{proof} 
\subsection{Dehn twist and homotopy}
In the proofs of Theorem \ref{asympt} and Theorem \ref{pants} we have used certain qualitative facts about Dehn twists. To recall the scenario let $\alpha$ be a simple closed geodesic on $S$ and let $\mathcal{C}_\alpha$ be the collar neighborhood around $\alpha$. Fix a set of Fermi coordinates on $\mathcal{C}_\alpha$ and orient $\alpha$ according to the orientation explained in \S 1.3.1.

Let $\eta, \xi$ be two end-to-end geodesic arcs. By Theorem \ref{endtoend} there is another end-to-end geodesic arc $\xi_0$ and $m \in \mathbb{Z}$ such that $\xi = \mathcal{D}^m_\alpha(\xi_0)$ with $\eta$ and $\xi_0$ are either identical or disjoint. Clearly $\iota(\eta, \xi) = m.$ Let $x$ be the point of intersection between $\alpha$ and $\eta$ and $y$ be the point of intersection between $\alpha$ and $\xi$. Consider the right half $R$ of $\mathcal{C}_\alpha \setminus \alpha$ with respect to the starting Fermi coordinates. Let $x_1, \cdots, x_k$ be the points of intersection between $\eta$ and $\xi$ that lies in $R$ arranged in the ascending order of their distances from $x$ along $\eta$. Let $\eta_i$ be the subarc of $\eta$ between $x$ and $x_i$ and $\xi_i$ be the subarc of $\xi$ between $y$ and $x_i$. Let $\eta_i: [0, 1] \to R$ and $\xi_i: [0, 1] \to R$ be their parametrization such that $\eta_i(0) = x_i = \xi_i(0)$ and $\eta_i(1)=x, \xi_i(1)=y$. 
\begin{lem}
There is a smooth homotopy $H: [0, 1] \times [0, 1] \to R$ between $\eta_i$ and $\xi_i$ such that: $H(s, 0) = \eta_i(s), H(s, 1) = \xi_i(s)$ and $H(1, t): [0, 1] \to \alpha$ is orientation reversing.
\end{lem} 
\begin{figure}[H]
 	\includegraphics[scale=.35]{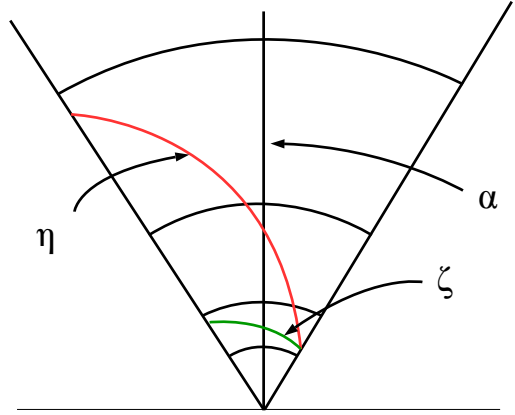}
 	\caption{Homotopy I}
 	\label{fig:homotopy1}
 \end{figure}
\begin{proof}
In a sense the above picture is our complete proof. The semi-annular regions are fundamental domains for $\mathcal{C}_\alpha$. In these fundamental domains we can explicitly draw lifts of any end-to-end geodesic arc. Namely, for $\mathcal{D}^m_\alpha(\xi_0)$ we would consider the two end points of $\xi_0$. Then we would use explicit expression for $D^m_\alpha(\xi_0)$ to draw one of its explicit lifts in $|m|$ consecutive fundamental domains of $\mathcal{C}_\alpha$. Finally to draw a lift of $\mathcal{D}^m_\alpha(\xi_0)$ explicitly we would recall that the latter is the geodesic (there is exactly one such) that joins the end points of the last lift of  $D^m_\alpha(\xi_0)$.
\end{proof}
Now we consider another Dehn twist considered in the proof of Theorem \ref{pants}. To explain our situation let us consider an $X$-piece. Let $\beta, \eta$ be the arcs as in picture. Consider the collar $\mathcal{C}_\eta$ around $\eta$ and fix a set of Fermi coordinates in it. Consider the orientation of $\eta$ determined by these coordinates as in \S 1.3.1. Observe that $\beta$ and $\mathcal{D}_\eta(\beta)$ are divided into two geodesic arcs by $\eta$. We shall consider the left half of these two arcs (determined by the Fermi coordinates). Let us denote these two arcs by $a_\beta$ and $a_{\mathcal{D}_\eta(\beta)}$.
\begin{lem}
There are parametrization $a_\beta: [0, 1] \to X, a_{\mathcal{D}_\eta(\beta)}: [0, 1] \to X$ and a smooth homotopy $H: [0, 1] \times [0, 1] \to X$ such that $H(s, 0) = a_\beta(s), H(s, 1) = a_{\mathcal{D}_\eta(\beta)}(s)$ and $H(0, t), H(1, t)$ are smooth maps from $[0, 1] \to \eta$. Moreover the last two maps are orientation preserving.
\end{lem}
\begin{figure}[H]
 	\includegraphics[scale=.35]{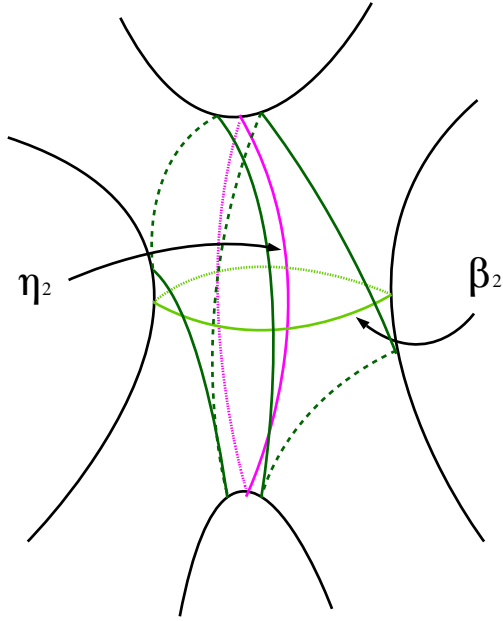}
 	\caption{Homotopy II}
 	\label{fig:homotopy2}
 \end{figure}
\begin{proof}
Consider the hyperbolic funnel $\mathbb{T}_\eta = \mathbb{H}^2/ \langle \gamma_\eta \rangle$ where $\gamma_\eta$ is a generator of $\pi_1(\mathcal{C}_\eta) \subset \pi_1(X)$. In particular, we can lift $\beta$ and $\mathcal{D}_\eta(\beta)$ on $\mathbb{T}_\eta$. Observe that the Fermi coordinates on $\mathcal{C}_\eta$ extends to a coordinate system on $\mathbb{T}_\eta$. With respect to these extended coordinates we consider cylindrical neighborhoods of $\eta$ that are bounded by curves equidistant from $\eta$ i.e. curves that look like $\{r = c \}$. Let $\mathcal{C}$ be the smallest such cylindrical neighborhood that contains the four intersections between lifts of $\beta$ and $\mathcal{D}_\eta(\beta)$ closest from $\eta$. Since Dehn twist is defined up to homotopy, it follows that each subarc of $\mathcal{D}_\eta(\beta)$ in $\mathcal{C}$ is the Dehn twist of a subarc of $\beta$ under the end point fixing homotopy. Hence the existence of our homotopy follows from a modified version of the last lemma.

\end{proof}


\begin{thebibliography}{}
\bibitem{Bu2}
P.\,Buser,
\emph{Geometry and spectra of compact Riemann surfaces.}
Reprint of the 1992 edition. Modern Birkh\"auser Classics.
Birkh\"auser, 2010. xvi+454 pp.


\bibitem{BuS}
P.\, Buser; K.\, D.\, Semmeler,
\emph{The geometry and spectrum of the one holed torus}
Comm. Math. Helv. 63(1998) 259-274.

\bibitem{B-K}
K.\, Burns; A.\, Kotok,
\emph{Manifolds with non-positive curvature.}
Ergod. Th. of Dynam. Sys. (1985), 5, 307-317.

\bibitem{C-M-E}
R.\,Canary; A.\, Marden; D.\, Epstein,
\emph{Fundamentals of hyperbolic manifolds},
Selected expositions-CUP (2006).

\bibitem{C}
C.\, Croke,
\emph{Rigidity for surfaces of nonpositive curvature.}
Comment. Math. Helv. 65 (1990), no. 1, 150–169.

\bibitem{F-M}
B.\, Farb; D.\ Margalit,
\emph{A primer on mapping class groups.}
Princeton Mathematical Series, 49.
Princeton University Press, Princeton, NJ, 2012. xiv+472 pp.

\bibitem{F-K}
R.\, Fricke; F.\, Klein,
\emph{Vorlesungen iiber die Theorie der Elliptischen Modulfunktionen/ Automorphen funktionen.}
G. Teubner: Leipzig, 1896/1912.

\bibitem{G}
I.\, M.\, Gel'fand,  
\emph{Automorphic functions and the theory of representations, Proc. Internat.}
Congress Math., (Stockholm, 1962), 74-85.

\bibitem{H}
A.\, HAAS, 
\emph{Length spectra as moduli for hyperbolic surfaces.} 
Duke Math. J., 52, (1985), 923-934.

\bibitem{Hu}
H.\, Huber,
\emph{Zur analytischen Theorie hyperbolischer Raumformen und Bewegungsgruppen I, II, Nachtrag zu II}
Math. Ann. 138 (1959), 1-26. Math. Ann. 142 (1961), 385-398; Math. Ann. 143 (1961), 463-464.

\bibitem{I}
N.\, V.\, Ivanov,
\emph{Subgroups of Teichmuller modular groups.}
Translated from the Russian by E. J. F. Primrose and revised by the author.
Translations of Mathematical Monographs, 115. American Mathematical Society, Providence, RI, 1992. xii+127 pp. ISBN: 0-8218-4594-2

\bibitem{K}
M.\, Kac, 
\emph{Can one hear the shape of a drum ?}
Amer. Math. Monthly 73 (1966), p. 1–23. MR0201237

\bibitem{M-M}
H.\, Masur; Y.\, Minsky,
\emph{Geometry of the complex of curves I: Hyperbolicity}
Invent. math. 138, 103–149 (1999)

\bibitem{M}
H.\ P.\,McKean,
\emph{Selberg's trace formula as applied to a compact Riemann surface.}
Comm. Pure Appl. Math. 25 (1972), 225–246.

\bibitem{O}
J-P.\, Otal,
\emph{Le spectre marqué des longueurs des surfaces à courbure négative. (French)
	[The marked spectrum of the lengths of surfaces with negative curvature]}
Ann. of Math. (2) 131 (1990), no. 1, 151–162.

\bibitem{P-S}
M.\, Pollicott; R.\, Sharp,
\emph{Angular self-intersections for closed geodesics
	on surfaces.}
Proc. of the AMS, Volume 134, Number 2, Pages 419–426 (2005).

\bibitem{Su}
T.\, Sunada,
\emph{Riemannian coverings and isospectral manifolds.}
Ann. of Math. (2) 121 (1985), no. 1, 169–186.

\bibitem{V}
M.\, F.\, Vign\'{e}ras,
\emph{Vari´et´es riemanniennes isospectrales et non isom´etriques.
(French)}
 Ann. of Math. (2) 112 (1980), no. 1, 2132.

\bibitem{W}
S.\,Wolpert,
\emph{The length spectra as moduli for compact Riemann surfaces.}
Ann. of Math. (2) 109 (1979), no. 2, 323–351.
\end{thebibliography}
\end{document}